\documentclass[aos,preprint]{imsart}
\RequirePackage[OT1]{fontenc}
\RequirePackage{natbib}
\RequirePackage{amsthm,amsmath}
\usepackage{pictexwd}
\usepackage{amssymb,graphicx,url}
\usepackage[colorlinks,citecolor=blue,urlcolor=blue]{hyperref}
\usepackage{bm}
\usepackage{arydshln}

\startlocaldefs    
\def\argmin{\mbox{argmin}}

\def\a{\alpha}
\def\b{\beta}
\def\R{\mathbb R}
\def\dd{\Delta}

\def\d{\delta}

\def\F{{\mathbb F}}
\def\G{{\mathbb G}}

\def\IK{{\mathbb K}}
\def\L{{\cal L}}

\def\P{{\mathbb P}}

\def\l{\lambda}
\def\labda1{\lambda_1}
\def\labda2{\lambda_2}

\def\e{\varepsilon}

\def\t{\tau}

\def\s{\sigma}

\def\argmin{\mbox{argmin}}

\def\comment#1{\relax}

\def\=in{\mathop{\rm =}}

\newtheorem{theorem}{Theorem}[section]

\newtheorem{lemma}{Lemma}[section]

\newtheorem{remark}{Remark}[section]

\numberwithin{equation}{section}
\theoremstyle{plain}
\setattribute{journal}{name}{}
\endlocaldefs

\usepackage{amsmath,here,amsfonts,latexsym,sym1,graphicx,here}
\usepackage{amsthm,curves}
\usepackage{graphicx}
\usepackage{caption}
\usepackage{subcaption}
\usepackage{float}

\def\th{\theta}

\def\P{{\mathbb P}}
\def\F{{\mathbb F}}
\def\G{{\mathbb G}}

\def\L{\Lambda}

\def\GG{\Gamma}

\begin{document}
\begin{frontmatter}
\title{Confidence intervals for the current status model}
\runtitle{Intervals for the current status model}
\begin{aug}
\author{\fnms{Piet} \snm{Groeneboom}\corref{}\ead[label=e1]{P.Groeneboom@tudelft.nl}
\ead[label=u1,url]{http://dutiosc.twi.tudelft.nl/\textasciitilde pietg/}}
and\ \
\author{\fnms{Kim} \snm{Hendrickx}\ead[label=e2]{kim.hendrickx@uhasselt.be}
\ead[label=u2,url]{http://www.uhasselt.be/fiche_en?voornaam=Kim&naam=HENDRICKX}}
\runauthor{P.\ Groeneboom and K.\ Hendrickx}
\affiliation{Delft University of Technology and Hasselt University}
\end{aug}

\begin{abstract}
	We discuss a new way of constructing pointwise confidence intervals for the distribution function in the current status model. The confidence intervals are based on the smoothed maximum likelihood estimator (SMLE) and constructed using bootstrap methods.
	Other methods to construct confidence intervals, using the non-standard limit distribution of the (restricted) MLE, are compared to our approach via simulations and real data applications.
\end{abstract}


\end{frontmatter}

\section{Introduction}  
Survival models are commonly used to characterize the distribution of a variable $X$ that is not observed directly.
Depending on what information is obtained on $X$, different censoring schemes arise. In this paper we consider the situation that a variable of interest is only known to lie before or after some random censoring variable $T$. Each observed sample consists of a set of $n$ inspection times $T_i$ (independent of the other $T_j$ and all $X_j's, j = 1,\ldots,n$)  and $n$ censoring indicators $\dd_i = 1_{\{X_i \leq T_i\}}$.
This type of censored data is known as current status data and arises naturally in reliability and survival studies when the status of an observational unit is only checked at one measurement point, which happens in especially when testing is destructive. One could say that the $i$th observation indicates the \textit{current status} of component $i$ at time $T_i$.
Estimation of the distribution function of the response variable in the current status model is harder than in right-censored models due to the lack of observing an actual event of interest.
\cite{GrWe:92} show that the (non-parametric) maximum likelihood estimator $\hat F_n$ (MLE), maximizing the likelihood of the data given by,
\begin{align} 
\label{likelihood}
\ell_n(F) = \sum_{i=1}^n \dd_i \log F(T_i)+(1-\dd_i)\log\{1-F(T_i)\},
\end{align}
over all possible distribution functions $F$ without making any additional constraints, converges pointwise at cube-root $n$ rate to the true distribution function $F_0$ of $X$. 
The Kaplan-Meier estimator (\citealp{Kaplan1958}), which is the MLE for right-censored data, converges on the contrary at a faster square root $n$ rate because of the fact that one has actual observations in addition to the censored ones. In the current status model all observations are censored.

In this paper we introduce new methods for constructing pointwise confidence intervals (CIs) for $F_0$ at time $t$ and compare our techniques with existing methods for interval estimation in current status models. We assume that both $X$ and $T$ have continuously differentiable distribution functions $F_0$ and $G$ respectively with positive derivatives $f_0$ and $g$ at $t$. From \cite{GrWe:92}, it is known that,
\begin{align}
\label{asumpt_distr_MLE}
n^{1/3} \left\{\hat F_n(t) - F_0(t)  \right\} \stackrel{{\cal D}}{\to} [4F_0(t)\{1-F_0(t)\}f_0(t)/g(t)]^{1/3}\mathbb{C},
\end{align}
where $\mathbb{C} = \arg\min_{t \in \R}\{ \mathbb{Z}(t)+ t^2\}$ and $\mathbb{Z}(t)$ is a standard two-sided Brownian motion process, originating from zero. To construct a confidence interval using the result given in (\ref{asumpt_distr_MLE}), we therefore need estimates of $f_0$ and $g$. If one is willing to make assumptions on the underlying distribution functions of $X$ and $T$, parametric methods can be used. This was done in e.g. \cite{Keiding:96} using Weibull models for both $f_0$ and $g$. Non-parametric estimates obtained by kernel smoothing were considered in \cite{banerjee_wellner:2005}. The choice of the tuning parameter is however crucial for a good performance of the confidence intervals. 

\cite{banerjee_wellner:2005} proposed a likelihood-ratio-based method for constructing pointwise confidence intervals for the distribution function in current status models. Starting from the likelihood ratio statistic
\begin{align*}
LR(\theta_0) = 2\left(\log \ell_n(\hat F_n)- \log \ell_n(\hat F_n^{\theta_0})  \right),
\end{align*}
for testing the null hypothesis $F_0(t) = \theta_0$, which has asymptotic distribution $\mathbb{D}$ characterized in \cite{mouli_jon:01}, the authors estimate the interval by
\begin{align*}
\left\{\theta \in (0,1) : LR(\theta) \leq d_{1-\a}   \right\},
\end{align*}
where $d_{1-\a}$ is the $(1-\a)$th percentile of $\mathbb{D}$. Here $\hat F_n$ denotes the unconstrained MLE maximizing (\ref{likelihood}) and $\hat F_n^{\theta_0}$ denotes the MLE of $F_0$ under the constrained that $F_0(t) = \theta_0$. 
The LR-based method avoids estimation of $f_0$ and $g$ since, under the null hypothesis, the limiting distribution is free of the underlying parameters. 
In contrast with the situation for the MLE itself and for the SMLE, no analytical information is available for the distribution $\mathbb{D}$ and the distribution $\mathbb{D}$ is estimated via simulations. A short proof of the characterization of the ``Chernoffian"  limit distribution of the MLE itself (without being restricted) in terms of Airy functions has recently been given in \cite{piet_nico_steve:15} and the asymptotic distribution of the SMLE is just normal. So in these cases tables of the critical values are available (for the MLE they are given in \cite{piet_jon:01}).
Tables to determine the asymptotic critical values for the LR test are available in \cite{mouli_jon:01}. 

More recently, bootstrap methods for constructing confidence intervals in the current status model have been considered. It is however proved in \cite{abrevaya_huang2005} that the naive bootstrap procedure, which simply resamples the original data will not work for pointwise confidence intervals for the distribution function $F_0$ if it is estimated by the MLE $\hat F_n$. A consistent model-based bootstrap procedure was introduced in \cite{SenXu2015}. Instead of resampling the $(T_i,\dd_i)$, the authors proposed resampling the $\dd_i$ from a Bernoulli distribution with success probability given by $\tilde F(T_i)$, where $\tilde F$ is an estimator of $F_0$ satisfying some smoothness conditions (that are not fulfilled by the ordinary MLE $\hat F_n$). The obtained bootstrap sample  $(T_1,\dd_1^*),\dots,(T_n,\dd_n^*)$ can next be used for interval estimation. In this case one computes the MLE $F_n^*$ in the bootstrap samples, and subtracts the smooth distribution $\tilde F$, generating the $\dd_i^*$. The confidence intervals are then formed by taking
$$
\left[\hat F_n(t)-V^*_{1-\a/2}(t),\hat F_n(t)-V^*_{\a/2}(t)\right],
$$
where $V_{\a}^*$ us the $\a$th quantile of $B$ values of
\begin{align*}
F_n^*(t)-\tilde F(t),
\end{align*}
where $B$ is the number of bootstrap samples taken. For current status and related models, some research has been reported recommending the use of this smooth bootstrap procedure. A smooth bootstrap calibration was used in \cite{durot2010} for a goodness-of-fit-test for monotone functions and in \cite{piet:11e} for a likelihood ratio type two-sample test for current status data. \cite{cecile_rik:12} used a similar approach to determine the critical value for testing equality of functions under monotonicity constraints. 
The main motivation for recommending the smooth bootstrap are the negative results by \cite{abrevaya_huang2005} and \cite{kosorok:08} proving the inconsistency of the naive bootstrap for generating the limiting distribution of the MLE. 

Recently, it was however proved in \cite{kim_piet:17} that the naive bootstrap of resampling with replacement from the data does works in case the underlying distribution function is estimated by the SMLE or in case interest is in other functionals than the values of the distribution function. The validity of the naive bootstrap for constructing pointwise confidence intervals around the SMLE and for doing inferences in the current status linear regression model \cite{GroeneboomHendrickx16} are illustrated in \cite{kim_piet:17}. Although \cite{durot2010} conjecture that the naive  bootstrap fails in their setting, this result suggests that this conjecture might be incorrect and that applications of the naive bootstrap involving the Grenander estimator are worthy of study in further research. 

Besides considering the naive or smooth bootstrap one could moreover consider resampling  the $\dd_i$ from the MLE itself. Simulation studies in \cite{cecile_rik:12} even suggest that the smooth bootstrap does not necessarily perform better than bootstrapping from the Grenander estimator in their setting.  So far, the theoretical properties of the latter bootstrap procedure remain an open problem. As a consequence of the positive result by \cite{kim_piet:17}, we conjecture that bootstrapping from the MLE might very well work for pointwise confidence intervals in the current status model, if one uses the right functional of the model as a basis for the intervals.

The outline of this paper is as follows. In Section \ref{section:pointwiseCI} we introduce the current status model, describe the construction of the Smoothed Maximum Likelihood estimator (SMLE) for the distribution function and explain how the smooth bootstrap procedure can be used to construct pointwise confidence intervals for the distribution function. The asymptotic behavior of the confidence intervals is also given in Section \ref{section:pointwiseCI} together with some details on how to improve the performance of our intervals. Simulation studies are reported in Section \ref{section:simulations} to demonstrate the finite sample behavior of our confidence intervals and to compare our method with existing methods proposed by \cite{banerjee_wellner:2005} and \cite{SenXu2015}. In Section \ref{section:realdata} we illustrate our methods on the Hepatitis A dataset and the Rubella dataset. Some concluding remarks are pointed out in Section \ref{section:concluding remarks}. An appendix is included in Section \ref{section:appendix} containing the proofs of our main results.

The proofs of our results are rather non-trivial and use techniques totally different from the
techniques used in \cite{banerjee_wellner:2005} and \cite{SenXu2015}. The latter fact is not unexpected, since the intervals are based on recently developed smooth functional theory (see, for example, \cite{piet_geurt:15}) and deal with asymptotically normal limits instead of the non-standard limits for the (restricted) MLE. We hope that the present paper serves the purpose of making these techniques more widely known. Rcpp scripts for all methods, discussed here (also the methods of \cite{banerjee_wellner:2005} and \cite{SenXu2015}) are available in \cite{github:15}.

\section{Pointwise confidence intervals in the current status model}
\label{section:pointwiseCI}
Consider an i.i.d sample $X_1,\ldots, X_n$ with distribution function $F_0$, where the distribution corresponding to $F_0$ has support $[0,M]$ and let $F_0$ have a density $f_0$ staying away from zero on $[0,M]$.
The observations in the current status model are $(T_1,\dd_1= 1_{\{X_1\leq T_1\}}),\ldots,(T_n,\dd_n= 1_{\{X_n\leq T_n\}})$ where the $T_i$ are independent of all $X_j's$ and have a distribution $G$ with Lebesgue density $g$ with a support that contains $[0,M]$. We assume that $g$ stays away from zero on $[0,M]$ and has a bounded derivative $g'$. In this section we develop a method for confidence interval estimation for $F_0(t)$ when $t$ is an interior point of $[0,M]$ and $f_0$ has a continuous derivative at $t$. We estimate $F_0(t)$ by the Smoothed Maximum Likelihood estimator (SMLE) obtained by first estimating the MLE $\hat F_n$ and then smoothing this using a smoothing kernel, i.e.,
\begin{align}
\label{SMLE}
\tilde F_{nh}(t)=\int\IK\left(\frac{t-x}{h}\right)\,d\hat F_n(x),
\end{align}
where $\IK$ is an integrated kernel,
\begin{align*}
\IK(u)=\int_{-\infty}^u K(x)\,dx,
\end{align*}
and where $h$ is a chosen bandwidth. Here  $d\hat F_n$ represents the jumps (``masses'') of the discrete distribution function $\hat F_n$ and $K$ is one of the usual kernels, used in density estimation (i.e. $K$ is a probability density with support $[-1,1]$ which is symmetric and twice continuously differentiable on $\R$). We use the notations $K_h$ and $\IK_h$ to denote the scaled versions of $K$ and $\IK$ respectively, given by:
\begin{align*}
K_h(u) = h^{-1}K(u/h) \quad \text{and} \quad \IK_h(u) = \IK(u/h). 
\end{align*} 
It is well-known that the MLE $\hat F_n$ can be characterized as the left continuous slope of the convex minorant of a cumulative sum diagram formed by the point $(0,0)$ and
\begin{align*}
\left(\sum_{j=1}^i w_j,\sum_{j=1}^i f_{1j}\right),\,i=1,\dots,m,
\end{align*}
where the $w_j$ are weights, given by the number of observations at point $T_{(j)}$, assuming that $T_{(1)}<\dots<T_{(m)}$ ($m$ being the number of different observations in the sample) are the order statistics of the sample $(T_1,\dd_1),\dots,(T_n,\dd_n)$ and where $f_{1j}$ is the number of $\dd_k$ equal to one at the $j$th order statistic of the sample. When no ties are present in the data (as is indeed the case in our simulations due to continuity assumptions of $g$, but is often not satisfied in real data examples), $w_j =1, m =n$ and $f_{1j} = \dd_{(j)}$, where $\dd_{(j)}$ corresponds to $T_{(j)}$.

From \cite{piet_geurt_birgit:10} (Theorem 4.2 p.\,365) it follows  that,
\begin{align*}
n^{2/5}\left\{\tilde F_{nh}(t)-F_0(t)\right\}\stackrel{{\cal D}}\longrightarrow N(\b,\s^2),
\end{align*}
where
\begin{align}
\label{mu-sigma}
\b=\frac{c^2f_0'(t)}{2}\int u^2 K(u)\,du \quad \text{and} \quad \s^2=\frac{F_0(t)\{1-F_0(t)\}}{cg(t)}\int K(u)^2\,du.
\end{align}

In the remainder of this Section we first introduce a procedure for interval estimation based on a smooth bootstrap resampling scheme and next elucidate some adjustments to improve the performance of the bootstrap confidence intervals. 

\subsection{The smooth bootstrap}
We obtain a bootstrap sample $(T_1,\dd_1^*),\dots,(T_n,\dd_n^*)$ by keeping the $T_i$ in the original sample fixed and by resampling the $\dd_i^*$ from a Bernoulli distribution with probability $\tilde F_{nh}(T_i)$. 
The following bootstrap $1-\a$ interval is suggested:
\begin{equation}
\label{CI_type1}
\left[\tilde F_{nh}(t)-U_{1-\a/2}^*(t),\tilde F_{nh}(t)-U_{\a/2}^*(t)\right],
\end{equation}
where $U_{\a}^*(t)$ is the $\a$th quantile of $B$ values of 
\begin{align*}
Z_{nh}(t)= \tilde F_{nh}^*(t)-\int \IK_h(t-u)\,d\tilde F_{nh}(u).
\end{align*}
Here $\tilde F_{nh}^*(t)$ is the SMLE in the bootstrap sample defined in the same way as in (\ref{SMLE}) but with $\hat F_n$ replaced by $\hat F_n^*$, i.e. the MLE in the bootstrap sample.

Under the model assumptions stated at the beginning of this section, we have the following main result showing that $n^{2/5}Z_{nh}(t)$ converges to a normal distribution with the same asymptotic variance as the SMLE. The proof of this result can be found in the Appendix. Some theoretical aspects of the bootstrap MLE $\hat F_n^*$, important for proving our main result, are given in Subsection \ref{section:asymptotic} below.
\begin{theorem}
	\label{th:bootstrap_SMLE}
	Let $h=h_n\sim cn^{-1/5}$, and let $\s^2$ be given by (\ref{mu-sigma}),then:
	\begin{align*}
	n^{2/5}\left\{\tilde F_{nh}^*(t)-\int \IK_h(t-u)\,d\tilde F_{nh}(u)\right\}\stackrel{{\cal D}}\longrightarrow N(0,\s^2),
	\end{align*}
	given the data $(T_1,\dd_1),\dots,(T_n,\dd_n)$,  almost surely along sequences $(T_1,\dd_1),(T_2,\dd_2),\dots$.
\end{theorem}

Note that we can write, 
\begin{align*}
\int \IK_h(t-u)\,d\tilde F_{nh}(u) &=\int \IK_h(t-u) \left\{\int  K_h(u-v) d\hat F_n(v) \right\}du.\\
&= \int  \int\IK((t-v)/h -w)K(w)\,dwd\hat F_n(v).
\end{align*}
In practice we therefore have to compute the convolution kernel $\widetilde{\IK}$, defined by:
\begin{align}
\label{def_tilde_IK}
\widetilde{\IK}(x) =  \int\IK(x -w)K(w)\,dw.
\end{align}
A picture of the functions $K$, $\IK$ and $\widetilde{\IK}$ is given in Figure \ref{fig:kernels} using the triweight kernel defined by:
\begin{align*}
K(u)=\frac{35}{32}\left(1-u^2\right)^31_{[-1,1]}(u).
\end{align*}

\begin{figure}[!ht]
	\centering
	\begin{subfigure}[b]{0.3\textwidth}
		\includegraphics[width=\textwidth]{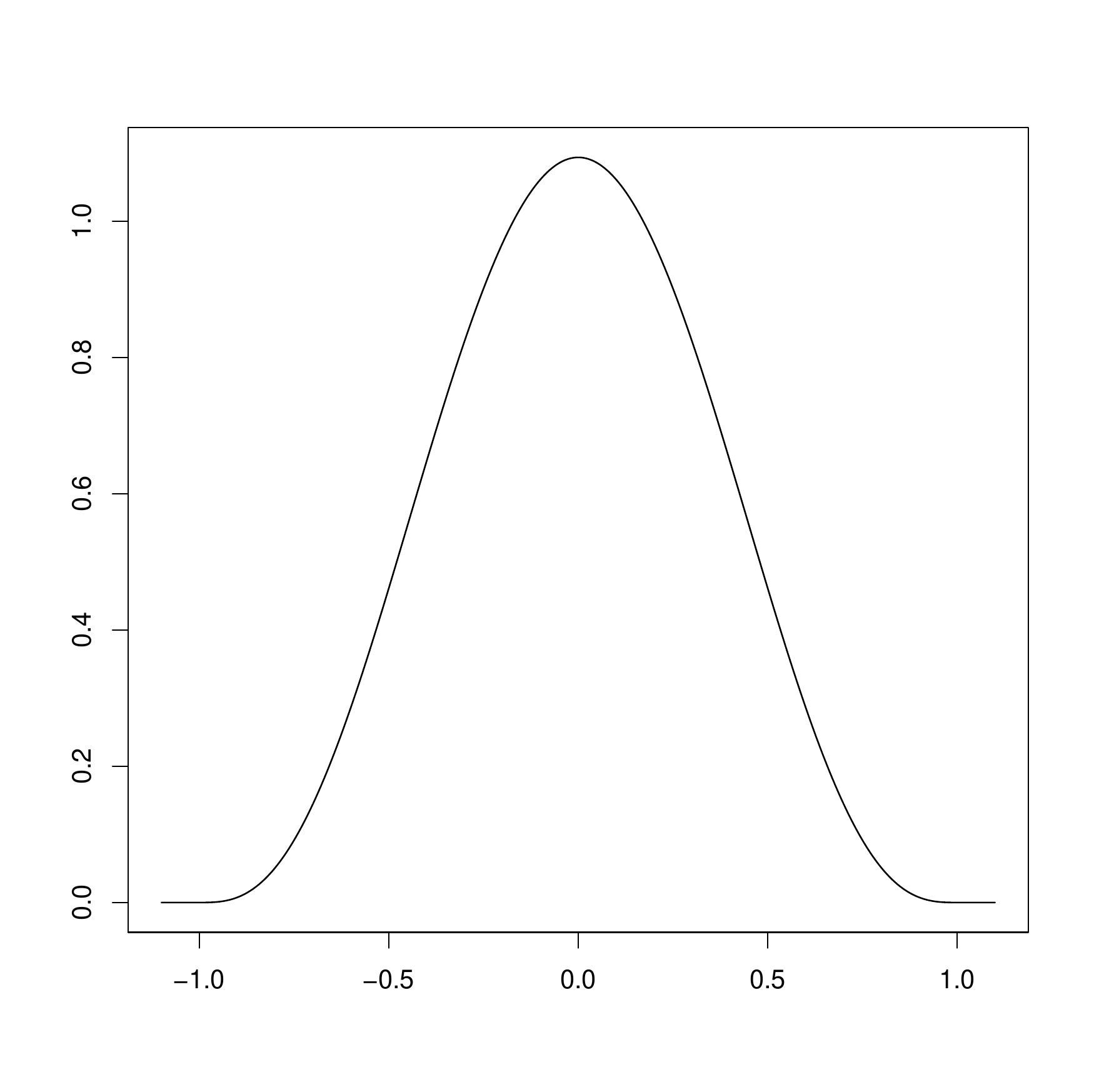}
		\caption{}
	\end{subfigure}
	\begin{subfigure}[b]{0.3\textwidth}
		\includegraphics[width=\textwidth]{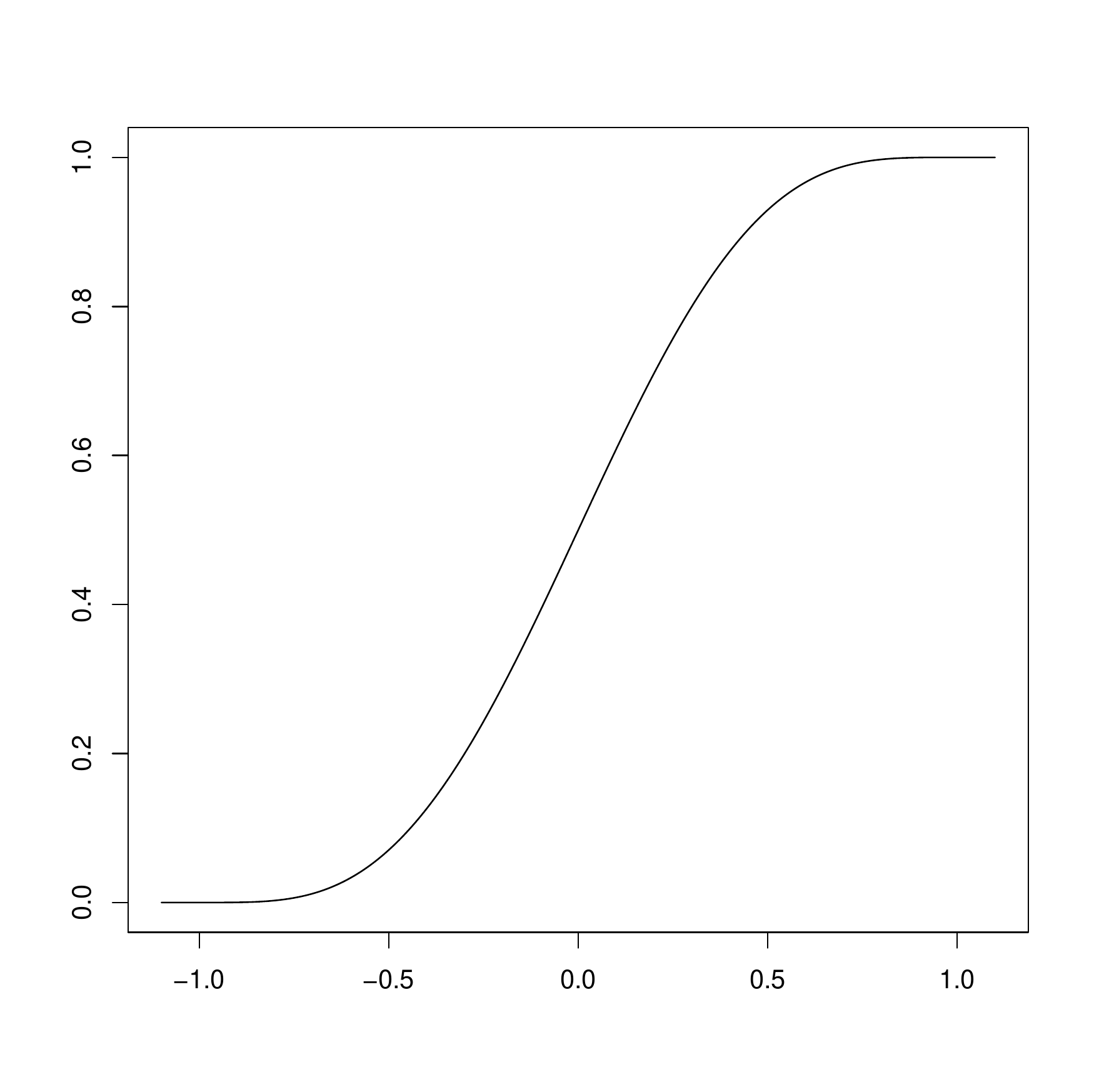}
		\caption{}
	\end{subfigure}
	\begin{subfigure}[b]{0.3\textwidth}
		\includegraphics[width=\textwidth]{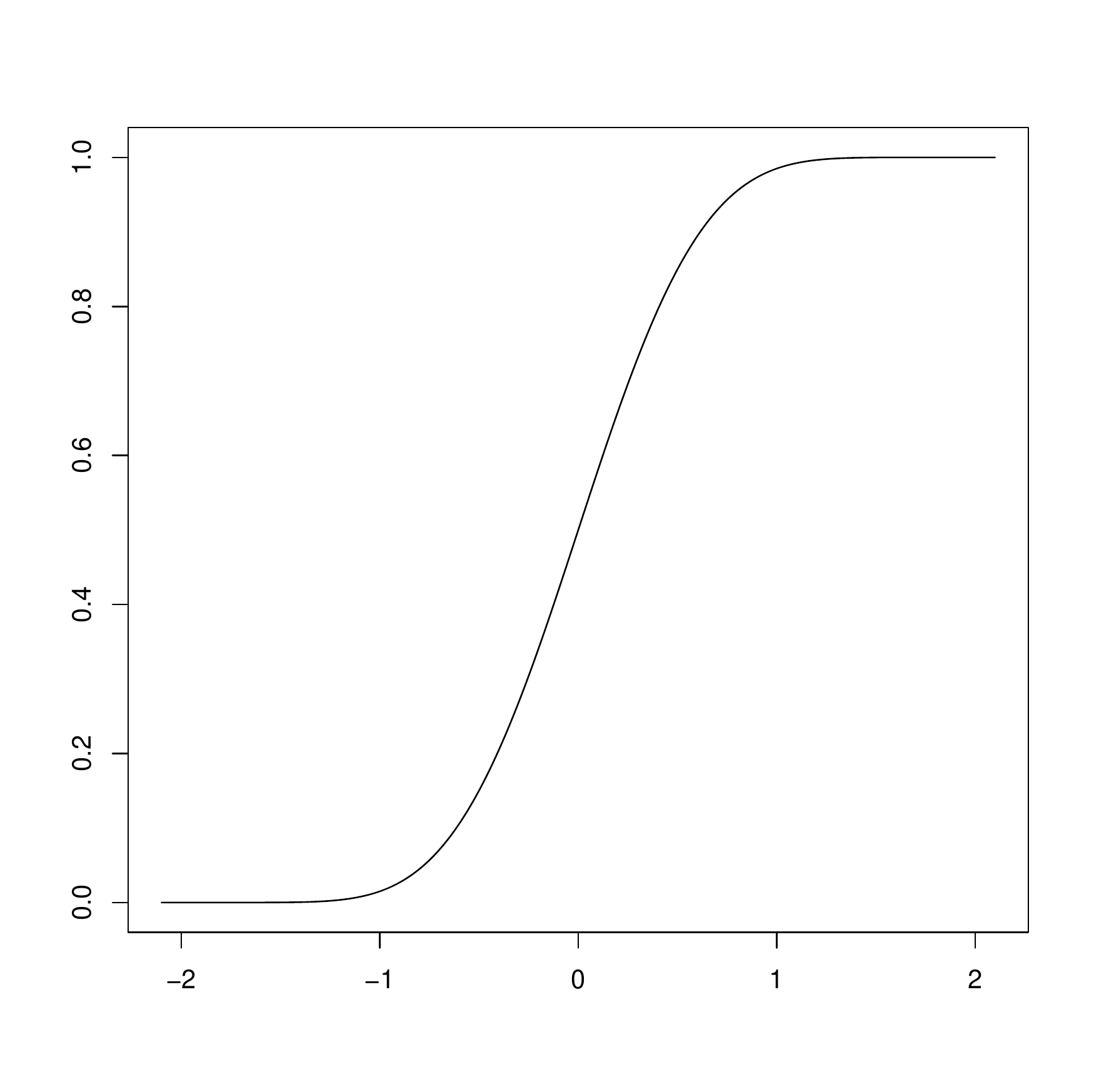}
		\caption{}
	\end{subfigure}
	\caption{(a) triweight kernel $K: x\mapsto \tfrac{35}{32}(1-x^2)^31_{[-1,1]}(x)\,$, (b) integrated triweight kernel $\IK: x \mapsto \int_{-\infty}^x K(w)\,dw$ and (c) the convolution $\widetilde{\IK}:x\mapsto \int \IK(x-w) K(w)\,dw$. }
	\label{fig:kernels}
\end{figure}

\begin{remark}{\rm Note that we subtract the integrated SMLE using the original data instead of the SMLE itself in the definition of $Z_{nh}(t)$ due to the bias of the SMLE $\tilde F_{nh}(t)$. This is in line with the method proposed by \cite{SenXu2015} where the authors subtract the SMLE instead of the MLE of the original data for constructing confidence intervals around the MLE. One needs to introduce an additional level of smoothing in order to construct valid intervals using the smooth bootstrap procedure. }
\end{remark}

\subsection{Asymptotic properties of the smooth bootstrap}
\label{section:asymptotic}
It is well-known that the $L_2$-distance between the MLE $\hat F_n$ in the original sample and the true distribution function $F_0$ is of order $n^{-1/3}$ (see e.g. \cite{geer:00} example 7.4.3). In the proof of Theorem \ref{th:bootstrap_SMLE} we need the following result
\begin{equation}
\label{local_L_2_bound1}
\int_{t-h}^{t+h}\bigl\{\hat F_n^*(x)-\tilde F_{nh}(x)\bigr\}^2\,dx= O_p^*\left(h n^{-2/3}\right),
\end{equation}
where $O_p^*\left(h n^{-2/3}\right)$ means that for all $\e>0$ and almost all sequences $(T_1,\dd_1),(T_2,\dd_2),\dots$, there exists an $M>0$ such that
\begin{align*}
P_n^*\left\{\int_{t-h}^{t+h}\bigl\{\hat F_n^*(x)-\tilde F_{nh}(x)\bigr\}^2\,dx \ge Mhn^{-2/3}\right\}<\e,
\end{align*}
for all large $n$. Here $P_n^*$ denotes the conditional probability measure given $(T_1,\dd_1),\ldots (T_n,\dd_n)$. Note that (\ref{local_L_2_bound1}) does not follow from a conditional global bound on the $L_2$-distance between the MLE $\hat F_n^*$ in the bootstrap sample and the SMLE $\tilde F_{nh}$ in the original sample of order $n^{-1/3}$ and that this is a refinement of the usual Hellinger distance calculations. 

By using the so-called ``switch-relation'' which reduces the study of the MLE $\hat F_n^*$ to the study of an inverse process (see e.g. \cite{piet_geurt:14} p. 320) we show in the Appendix that 
\begin{align}
\label{L2_bound_expectation}
E_n^*\left\{\hat F_n^*(t)-\tilde F_{nh}(t)\right\}^2 \leq Kn^{-2/3} \quad \forall t \in [0,M],
\end{align}
where $E_n^*$ denotes the conditional expectation given $(T_1,\dd_1),\ldots (T_n,\dd_n)$. From this result it follows that (\ref{local_L_2_bound1}) holds.

In the remainder of this section we describe techniques to improve the confidence intervals defined in (\ref{CI_type1}) by (a) considering estimation of the variance, (b) taking into account the boundary effects of kernel estimates and (c) estimating the asymptotic bias $\b$ defined in (\ref{mu-sigma}).

\subsection{Studentized confidence intervals}
Usually the performance of the bootstrap confidence intervals works best if one uses a pivot, obtained by Studentizing. In each bootstrap sample we therefore estimate the variance $\s^2$ defined in (\ref{mu-sigma}), apart from the factor $cg(t)$, which drops out in the Studentized bootstrap procedure, by,
\begin{align}
\label{S*}
S_{nh}^*(t)=n^{-2}\sum_{i=1}^n K_h(t-T_i)^2\left(\dd_i^*-\hat F_n^*(T_i)\right)^2.
\end{align}
The variance estimate defined in (\ref{S*}) is inspired by the fact that the SMLE $\tilde F_{nh}$ is asymptotically equivalent to the toy estimator,
\begin{align*}
\tilde F_{nh}^{toy}(t) = \int\IK_h(t-x)\,d F_0(x) + \frac1n \sum_{i=1}^{n} \frac{K_h(t-T_i)\{\dd_i -F_0(T_i)\}^2}{g(T_i)}, 
\end{align*}
which has sample variance
\begin{align*}
S_n(t)=\frac1{n^{2}}\sum_{i=1}^n\frac{ K_h(t-T_i)^2\left(\dd_i- F_0(T_i)\right)^2}{g(T_i)^2}.
\end{align*}
We next compute
\begin{align*}
W_{nh}^*(t)=\frac{\tilde F_{nh}^*(t)-\int \IK_h(t-u)\,d\tilde F_{nh}(u)}{\sqrt{S_{nh}^*(t)}}\,.
\end{align*}
Let $Q_{\a}^*(t)$ be the $\a$th quantile of $B$ values of $W_{nh}^*(t)$, where $B$ is the number of bootstrap samples. Then the following bootstrap $1-\a$ interval is suggested:
\begin{equation}
\label{CI_type2}
\left[\tilde F_{nh}(t)-Q_{1-\a/2}^*(t)\sqrt{S_{nh}(t)},
\tilde F_{nh}(t)-Q_{\a/2}^*(t)\sqrt{S_{nh}(t)}\right],
\end{equation}
where $S_{nh}(t)$ is the variance estimate in the original sample obtained by replacing $\dd_i^*-\hat F_n^*(T_i)$ in (\ref{S*}) by $\dd_i-\hat F_n(T_i)$. Note that we do not need an estimate of the density $g$ in each of the observations $T_i$ as a consequence of the fact that $g(u)$ is close to $g(t)$ for $u \in [t-h,t+h]$. If, on the contrary, one wants to consider Wald-type confidence intervals for the distribution function based on the asymptotic normality results of the SMLE, estimation of $g$ is inevitable.

\subsection{Boundary correction}
\label{subsec:boundary-correction}
It is well-known that kernel density and distribution estimators without boundary correction are generally inconsistent at the boundary of the support $[0,M]$. We therefore use the boundary correction method proposed in \cite{piet_geurt:14}, and define the SMLE as
\begin{align}
\label{SMLE_corr}
\tilde F_{nh}^{(bc)}(t)=\int\left\{\IK\left(\frac{t-x}{h}\right) + \IK\left(\frac{t+x}{h}\right) - \IK\left(\frac{2M-t-x}{h}\right)  \right\}\,d\hat F_n(x).
\end{align}
The boundary corrected version of $Z_{nh}^*(t)$ is defined by:
\begin{align*}
\label{SMLE_integ_corr}
Z_{nh}^{(bc)*}(t) &= \tilde F_{nh}^{(bc)*}(t) - \int\left\{\IK_h(t-x) + \IK_h(t+x) - \IK_h(2M-t-x)  \right\}\,d\tilde F_{nh}^{(bc)}(x).
\end{align*}
The result of Theorem \ref{th:bootstrap_SMLE} remains valid under this boundary correction. We also have the following lemma.
\begin{lemma}
	\label{lemma:boundary_correction}
	Let the boundary corrected estimate $\tilde F_{nh}^{(bc)}$ be defined by (\ref{SMLE_corr}), and let $\widetilde{\IK}_h$ be defined by:
	\begin{align*}
	\widetilde{\IK}_h(u)=\widetilde{\IK}(u/h),\qquad u\in\R.
	\end{align*}
	where the convolution kernel $\widetilde{\IK}$ is defined by (\ref{def_tilde_IK}). Moreover, let $0<h\le M/3$. Then:
	\begin{align}
	&\int \left\{\IK_h(t-x)+\IK_h(t+x)-\IK_h(2M-t-x)\right\}\,d\tilde F_{nh}^{(bc)}(x)\nonumber\\
	&=\int\left\{\widetilde{\IK}_h(t-x) + \widetilde{\IK}_h(t+x) - \widetilde{\IK}_h(2M-t-x) \right\}\,d\hat F_n(x).
	\end{align}
\end{lemma}
From Lemma \ref{lemma:boundary_correction}, it follows that we can write,
\begin{align*}
Z_{nh}^{(bc)*}(t) 
&=\int\left\{\IK_h(t-x) + \IK_h(t+x) - \IK_h(2M-t-x)  \right\}\,d\big(\hat F_n^*-\tilde F_{nh}^{(bc)}\bigr)(x)\\
&=\int \left\{\IK_h(t-x) + \IK_h(t+x) - \IK_h(2M-t-x)  \right\}d\hat F_n^*(x) \\
&\qquad - \int \left\{\widetilde{\IK}_h(t-x) + \widetilde{\IK}_h(t+x) - \widetilde{\IK}_h(2M-t-x)  \right\}d\hat F_n(x)
\end{align*}

The proof of Lemma \ref{lemma:boundary_correction} is given in the Appendix. A picture of the MLE, together with the SMLE, both corrected and uncorrected for boundary effects is shown in Figure \ref{fig:SMLE-exponential}(a) for a sample from the truncated exponential distribution on [0,2]  (See Section \ref{section:simulations} for a detailed description of the model). Figure \ref{fig:SMLE-exponential}(b) presents the boundary corrected and uncorrected integrated SMLE and clearly shows the improvement of the boundary correction.

\begin{figure}[!ht]
	\centering
	\begin{subfigure}[b]{0.35\textwidth}
		\includegraphics[width=\textwidth]{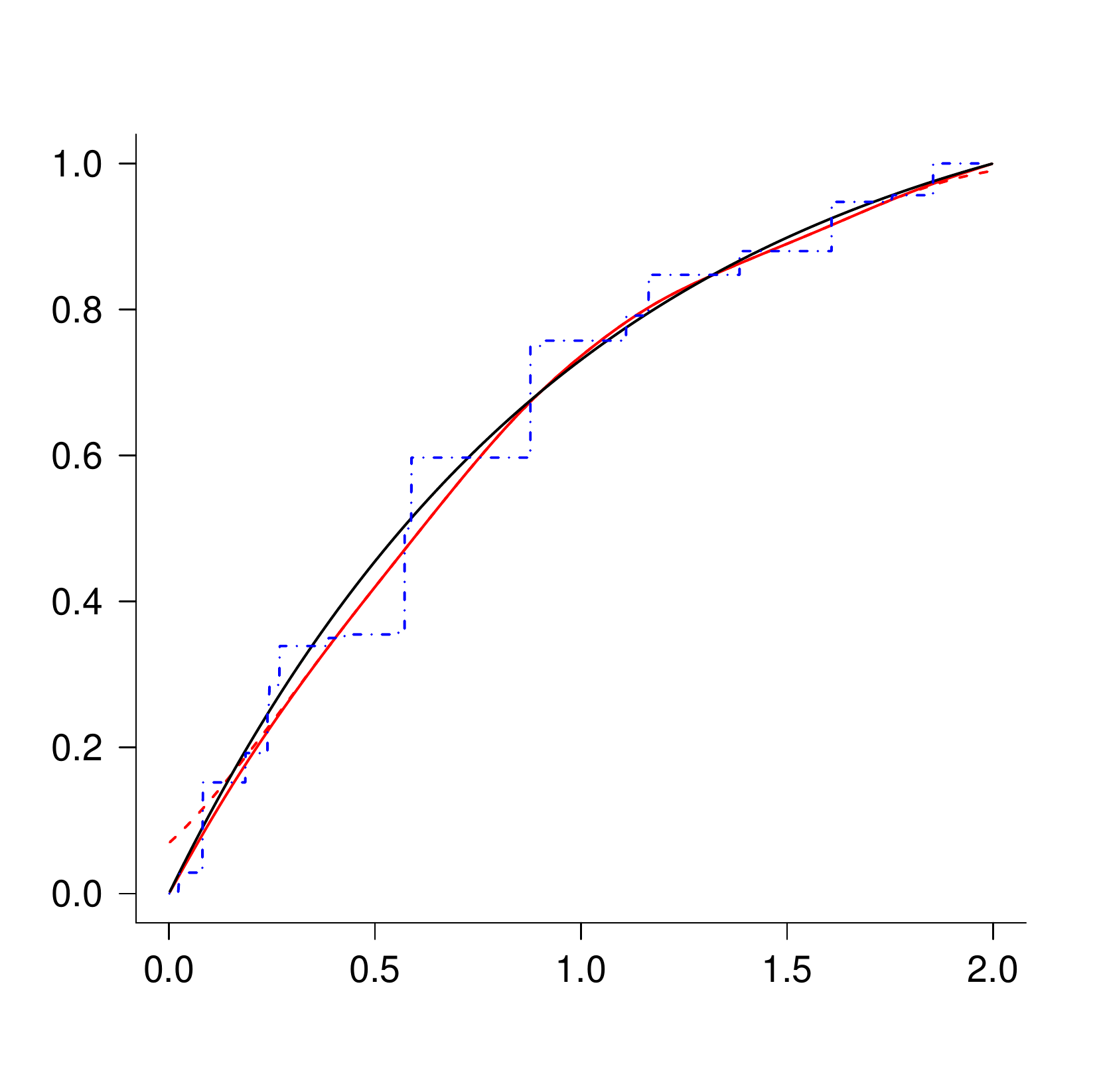}
		\caption{}
		\label{fig:SMLE-exponential1}
	\end{subfigure}
	\hspace{1cm}
	\begin{subfigure}[b]{0.35\textwidth}
		\includegraphics[width=\textwidth]{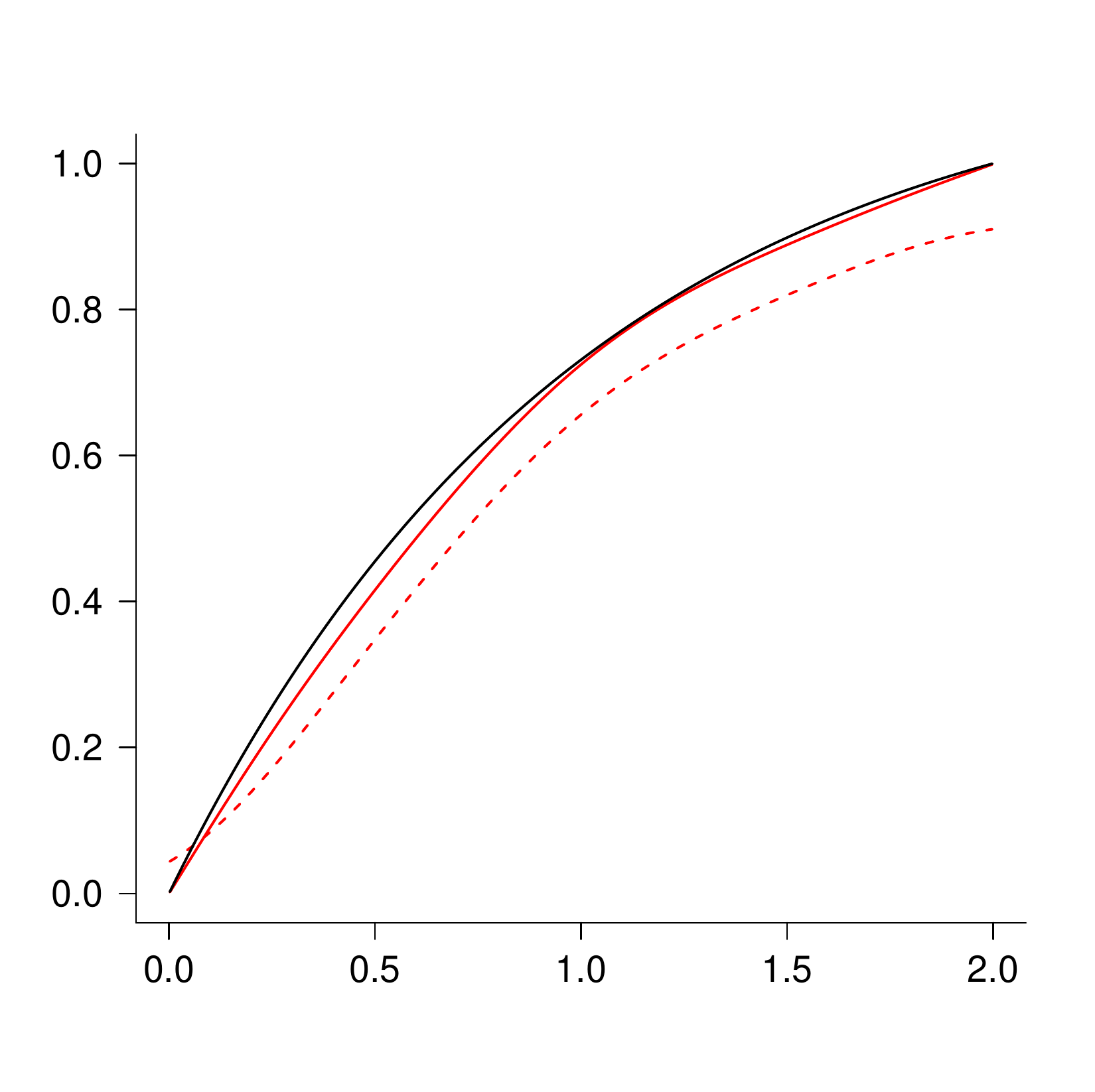}
		\caption{}
		\label{fig:SMLE-exponential-integrated}
	\end{subfigure}	
	\caption{Truncated exponential samples: (a) $F_0$ (black, solid), the MLE (blue,dashed-dotted), the SMLE with boundary correction (red, solid) and the SMLE without boundary correction (red, dashed) and (b) $F_0$ (black, solid), the integrated SMLE with boundary correction (red, solid) and the integrated SMLE without boundary correction (red, dashed); $n=1000$ and $h=2n^{-1/5}$.}
	\label{fig:SMLE-exponential}
\end{figure}

\subsection{Bias estimation}
When constructing confidence intervals around the SMLE, one should take into account the bias of the SMLE. Note that this matter does not occur for confidence intervals around the MLE, as in \cite{SenXu2015}, since the asymptotic distribution of the MLE is symmetric around zero. Direct estimation of the asymptotic bias $\b$ defined in (\ref{mu-sigma}) requires a consistent estimate of the second derivative $f_0'$ of the distribution function $F_0$. Although it is possible to estimate $f_0'$ consistently (see e.g.\cite{piet_geurt:15} p.\ 243), our computer experiments demonstrated that it is very difficult to estimate the bias term sufficiently accurately. We therefore propose to use an adaptive bandwidth $h=h(t)$ in order to improve the performance of the SMLE-based CIs. When $h$ is monotone increasing and continuous in $t$, then the SMLE $\tilde F_{nh(t)}$ is also monotone increasing and continuous in $t$. The effect of the adaptive bandwidth will be elucidated further in Section \ref{section:simulations}.

\section{Simulations}
\label{section:simulations}

In this section we illustrate the finite sample behavior of the SMLE-based CIs introduced in Section \ref{section:pointwiseCI}. We also compare our CIs with the likelihood-ratio based CIs proposed by \cite{banerjee_wellner:2005} and the CIs introduced in \cite{SenXu2015}. The latter CIs are both constructed around the MLE. We use two simulation examples to analyze the effect of Studentizing and the choice of the kernel $K$ on the behavior of our SMLE-based CIs. We propose a criterion for bandwidth selection and illustrate how undersmoothing the bandwidth can improve the behavior of the CIs.

In the first simulation setting both the event times and censoring times are sampled from a Uniform(0,2)-distribution. Since the derivative of the uniform density equals zero, the SMLE is an unbiased estimate of the uniform distribution function and no bias correction is needed. 
In the next simulation set-up, we consider the model where we generate the event times from a truncated exponential distribution on [0,2] and take Uniform(0,2)-censoring times.

For sample sizes $n=100,500,1000$ and $2000$ we generated 5000 data sets from both models. The boundary correction described in Section \ref{subsec:boundary-correction}, is used each time the SMLE is considered. The number of bootstrap samples within each simulation run equals $B=1000$.

Table \ref{table:uniform1} shows the coverage percentage, i.e. the number of times (out of the 5000 simulation runs) that $F_0(t)$ is not in the 95\% CIs,  and the average length of the 95\% CIs around $F_0(t)$ for the uniform model and $t=1$. We use the bandwidth $h=cn^{-1/5}$, where the constant $c=2.0$ corresponds to the length of the interval $[0,2]$. We consider two different choices for the kernel, the triweight kernel and the Epanechnikov kernel and compare the results of our SMLE-based CIs (\ref{CI_type2}) with the results for the MLE-based methods of \cite{banerjee_wellner:2005} and \cite{SenXu2015}.

\begin{table}[h]
	\caption{Uniform samples}
	\begin{tabular}{r|cc:cc|cc|cc:cc}
		\hline
		& \multicolumn{4}{c}{Studentized SMLE-based CI (\ref{CI_type2})}	& \multicolumn{2}{|c|}{Banerjee-Wellner}	& \multicolumn{4}{c}{Sen-Xu}\\	
		& \multicolumn{2}{c}{\textit{Triweight}} & \multicolumn{2}{:c}{\textit{Epanechnikov}} 	& \multicolumn{2}{|c|}{ }	& \multicolumn{2}{c}{\textit{Triweight}} & \multicolumn{2}{:c}{ \textit{Epanechnikov}} 	\\
		$n$  &CP&L&CP&L&CP&L&CP&L&CP&L\\
		\hline
		100   & 0.0326& 0.2799& 0.0358 & 0.2376 & 0.0486& 0.3897& 0.0568& 0.4620 & 0.0470 & 0.4625 \\
		500   & 0.0472& 0.1473& 0.0454 &0.1276 &  0.0504& 0.2311& 0.0636&0.2532 &0.0580 & 0.2536\\
		1000  & 0.0626& 0.1072& 0.0600& 0.0928 & 0.0498& 0.1846& 0.0654& 0.2024 & 0.0596&0.2028 \\
		2000  & 0.0494& 0.0827& 0.0502 & 0.0710 & 0.0414& 0.1466& 0.0516& 0.1598 & 0.0482 & 0.1599 \\
		\hline
		\multicolumn{11}{c}{CP: Coverage proportion, L = average length ($\a = 0.05$).  }
	\end{tabular}
	\label{table:uniform1}
\end{table}

For each point $ t_i=0.02,0.04,\dots,2$ Figure \ref{fig:uniform_kernels}(a) presents the coverage proportions for the Studentized SMLE-based CIs (\ref{CI_type2}) using the Epanechnikov kernel and the triweight kernel and illustrates that the choice of the kernel has only a small effect on the coverage proportions. The average length of the CIs, shown in Figure \ref{fig:uniform_kernels}(b) is smaller for the intervals constructed with the Epanechnikov kernel. 
\begin{figure}[!ht]
	\centering
	\begin{subfigure}[b]{0.35\textwidth}
		\includegraphics[width=\textwidth]{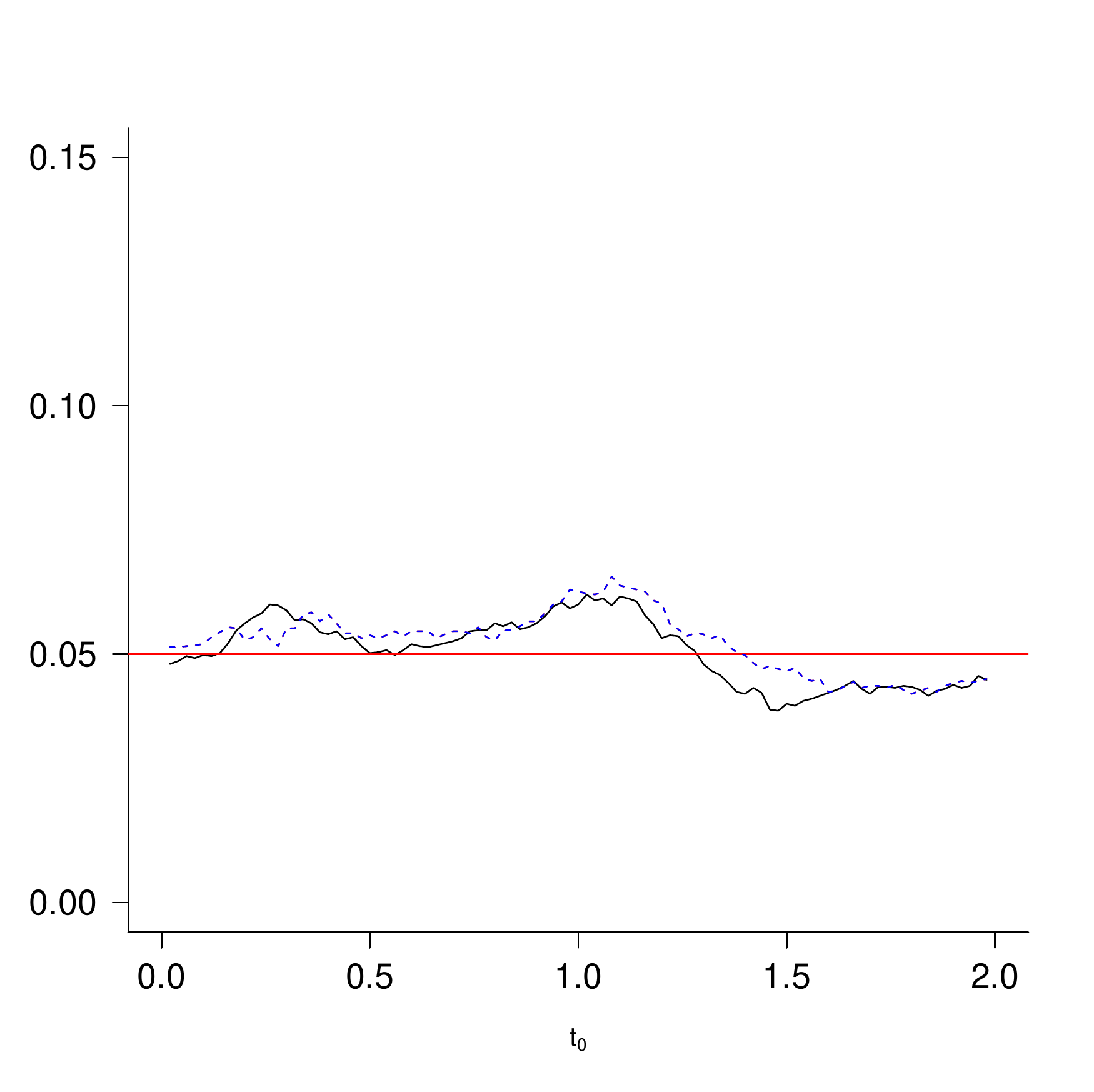}
		\caption{}
	\end{subfigure}
	\hspace{1cm}
	\begin{subfigure}[b]{0.35\textwidth}
		\includegraphics[width=\textwidth]{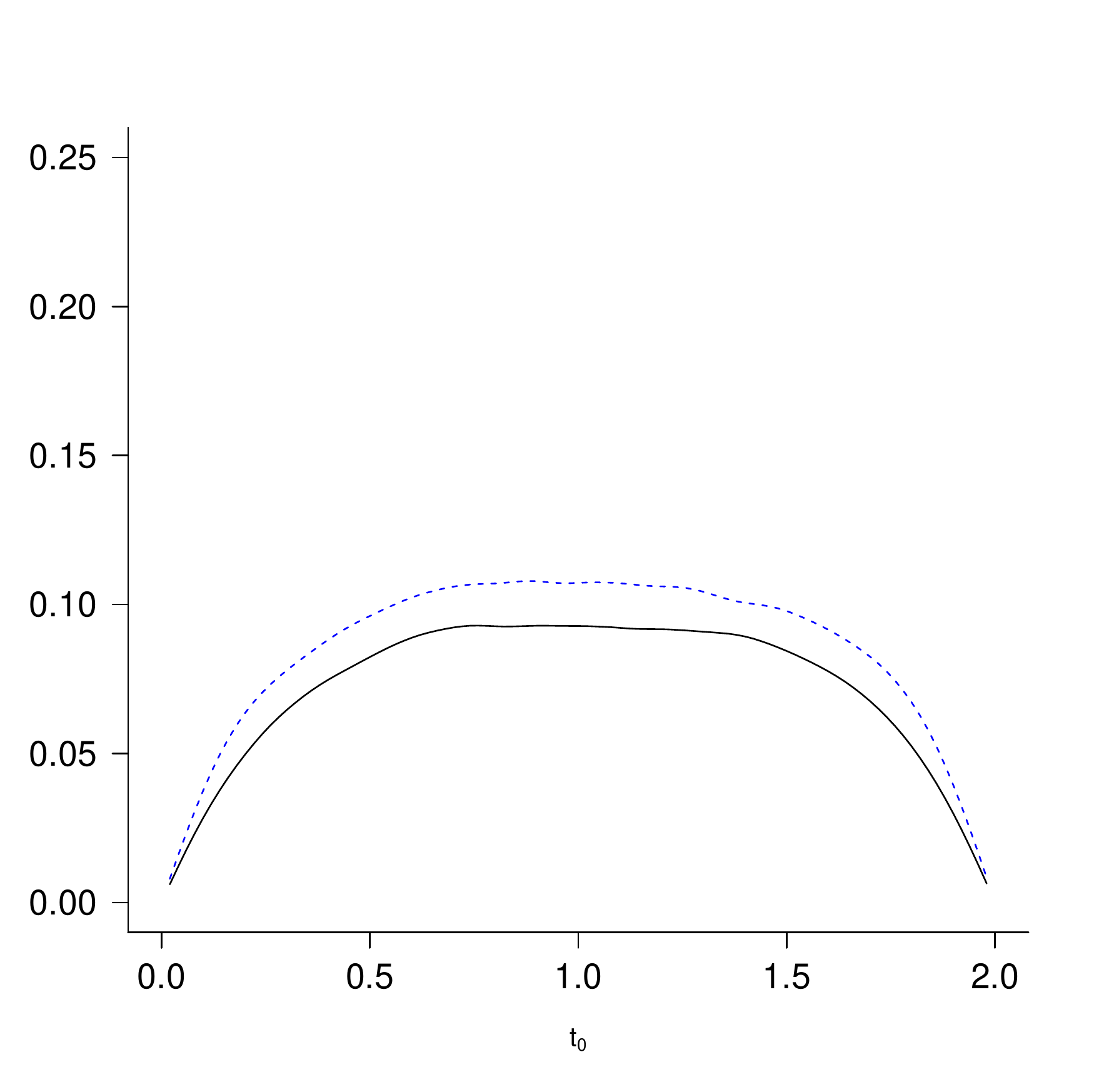}
		\caption{}
	\end{subfigure}
	\caption{Uniform samples: (a) Proportion of times that $F_0(t_i),\, t_i=0.02,0.04,\dots$ is not in the $95\%$ CI's and (b) average length of the CIs in $5000$ samples using $1000$ bootstrap samples with the Epanechnikov kernel (black,solid) and the triweight kernel (blue, dashed) for the Studentized SMLE-based CIs (\ref{CI_type2}). $h=2n^{-1/5}$.}
	\label{fig:uniform_kernels}
\end{figure}

A picture of the proportion of times that $F_0(t_i), t_i = 0.02,0.04,\ldots,2$ is not in the 95\% CIs for $n=1000$ is shown in Figure \ref{fig:uniform1}(a-c) for the uniform model. The average length of our CIs based on the SMLE (both classical CIs (\ref{CI_type1}) (result not shown) and Studentized CIs (\ref{CI_type2})) remains smaller than the average lengths of the Banerjee-Wellner and Sen-Xu CIs based on the MLE for all points $t$, as is shown in Figure \ref{fig:uniform1}(d). For the uniform samples our SMLE-based method does not suffer from bias effects; the coverage of the different intervals is comparable for time points in the middle of the interval [0,2], but becomes rather bad at the boundary of the interval for the Banerjee-Wellner and Sen-Xu intervals. Figure \ref{fig:uniform1} is obtained with the results for the Epanechnikov kernel. Similar comparisons were obtained when the triweight kernel was used. 
Figure \ref{fig:uniform1}(a) also shows that the classical SMLE-based CIs (\ref{CI_type1}) are slightly anti-conservative near the left boundary of the interval and have a coverage that is less good than the Studentized CIs (\ref{CI_type2}). Similar conclusions are also observed for the exponential samples. 
\begin{figure}[!ht]
	\centering
	\begin{subfigure}[b]{0.35\textwidth}
		\includegraphics[width=\textwidth]{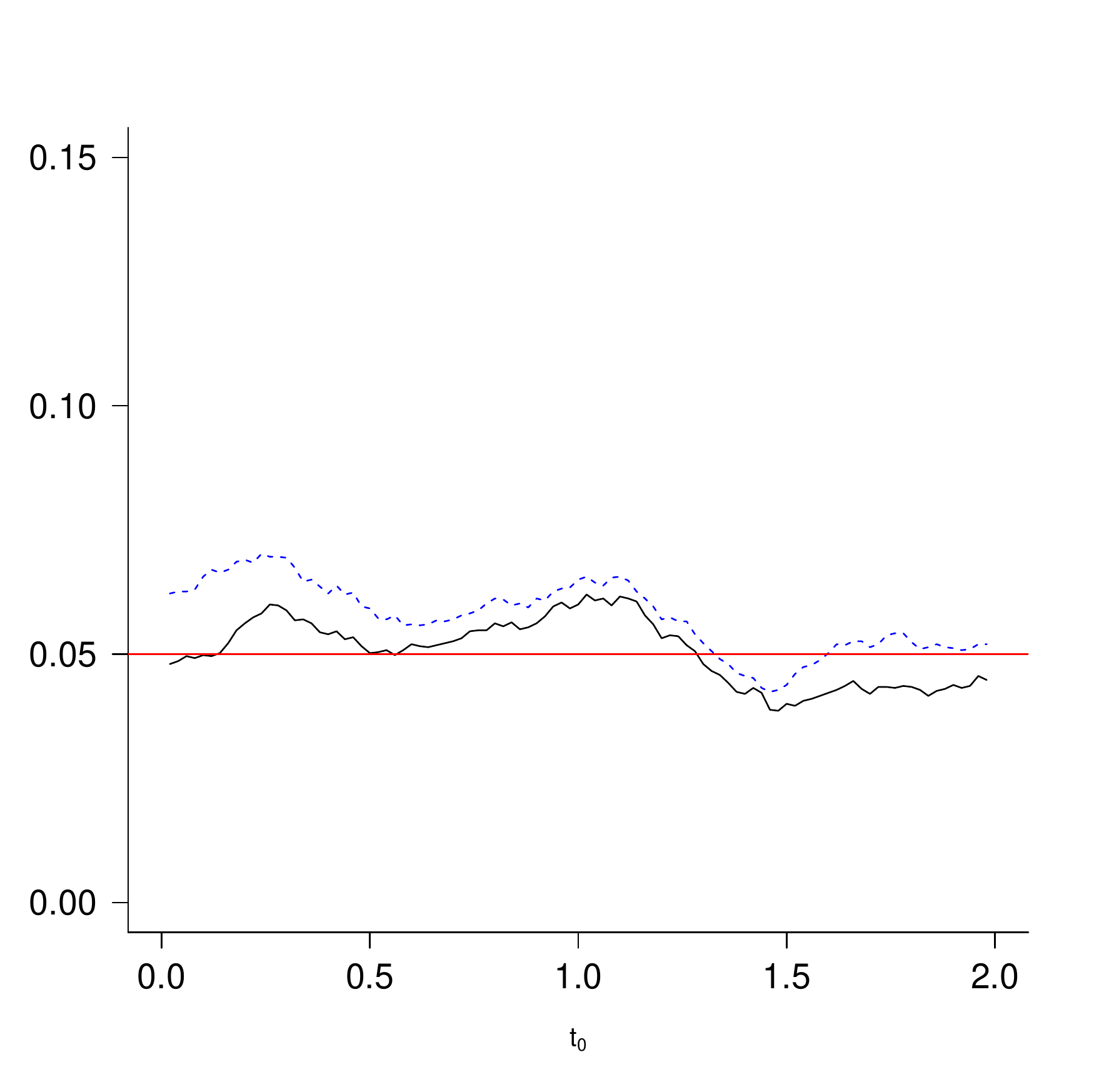}
		\caption{}
	\end{subfigure}
	\hspace{1cm}
	\begin{subfigure}[b]{0.35\textwidth}
		\includegraphics[width=\textwidth]{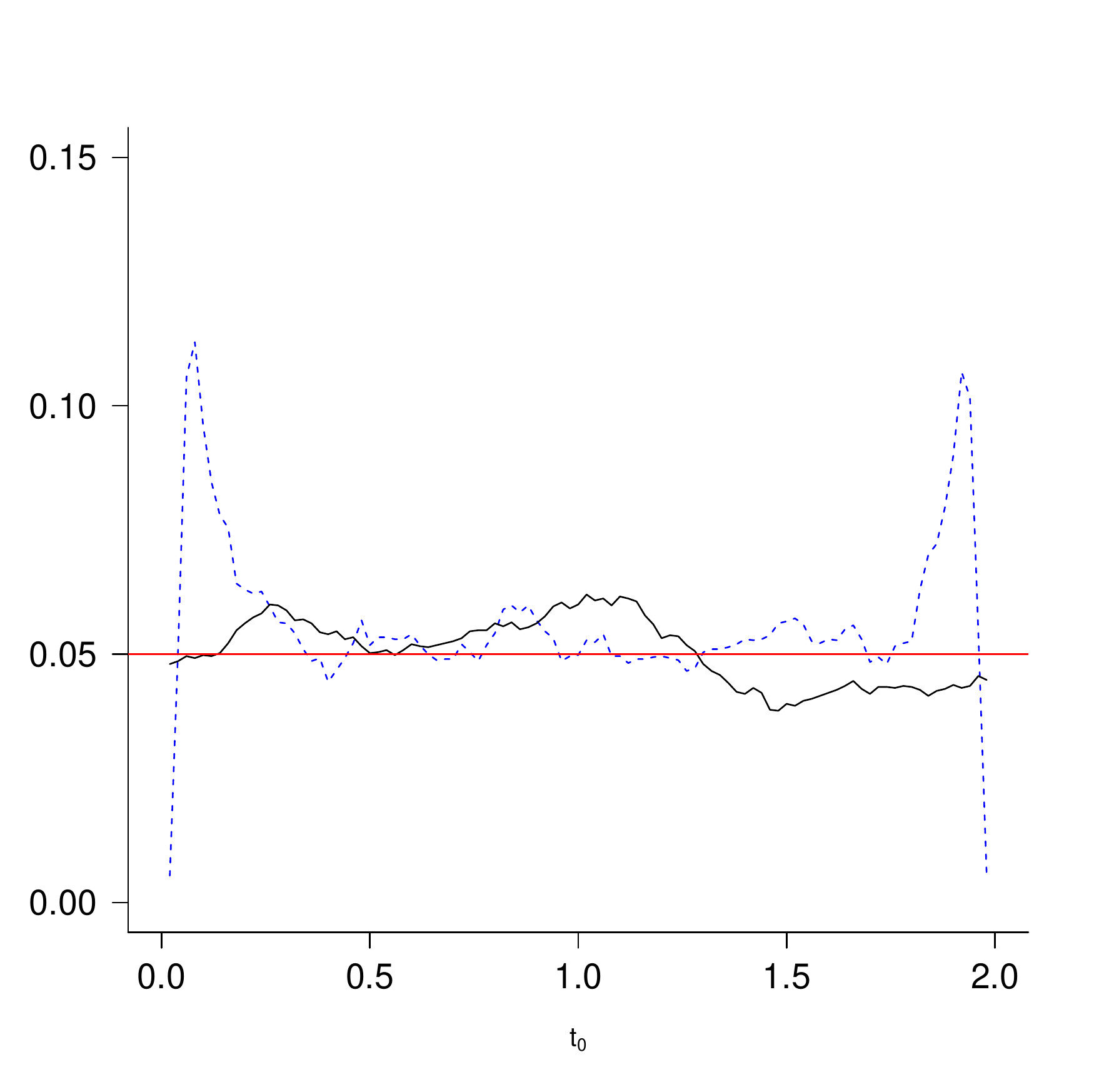}
		\caption{}
	\end{subfigure}	\\
	\begin{subfigure}[b]{0.35\textwidth}
		\includegraphics[width=\textwidth]{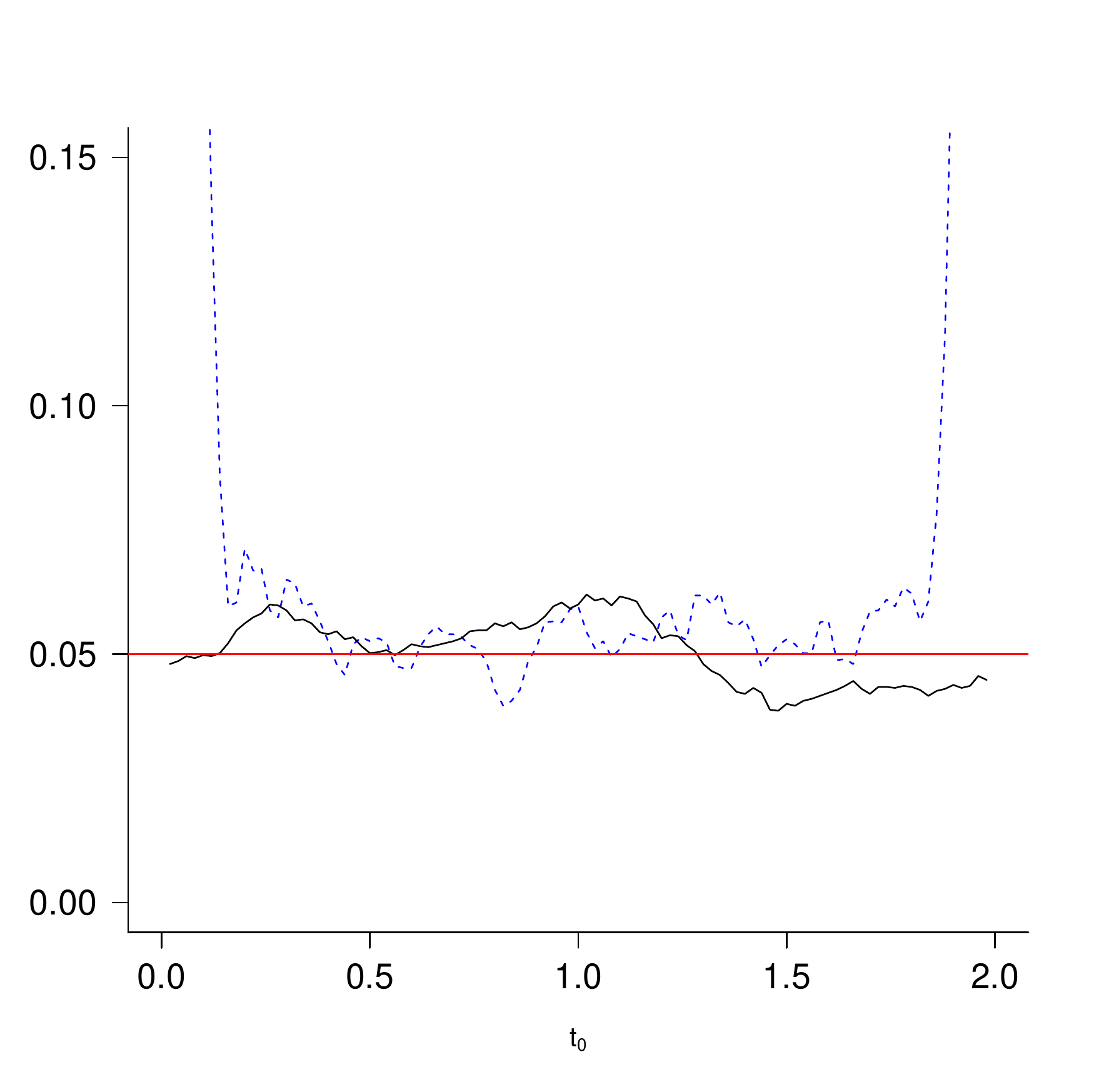}
		\caption{}
	\end{subfigure}
	\hspace{1cm}
	\begin{subfigure}[b]{0.35\textwidth}
		\includegraphics[width=\textwidth]{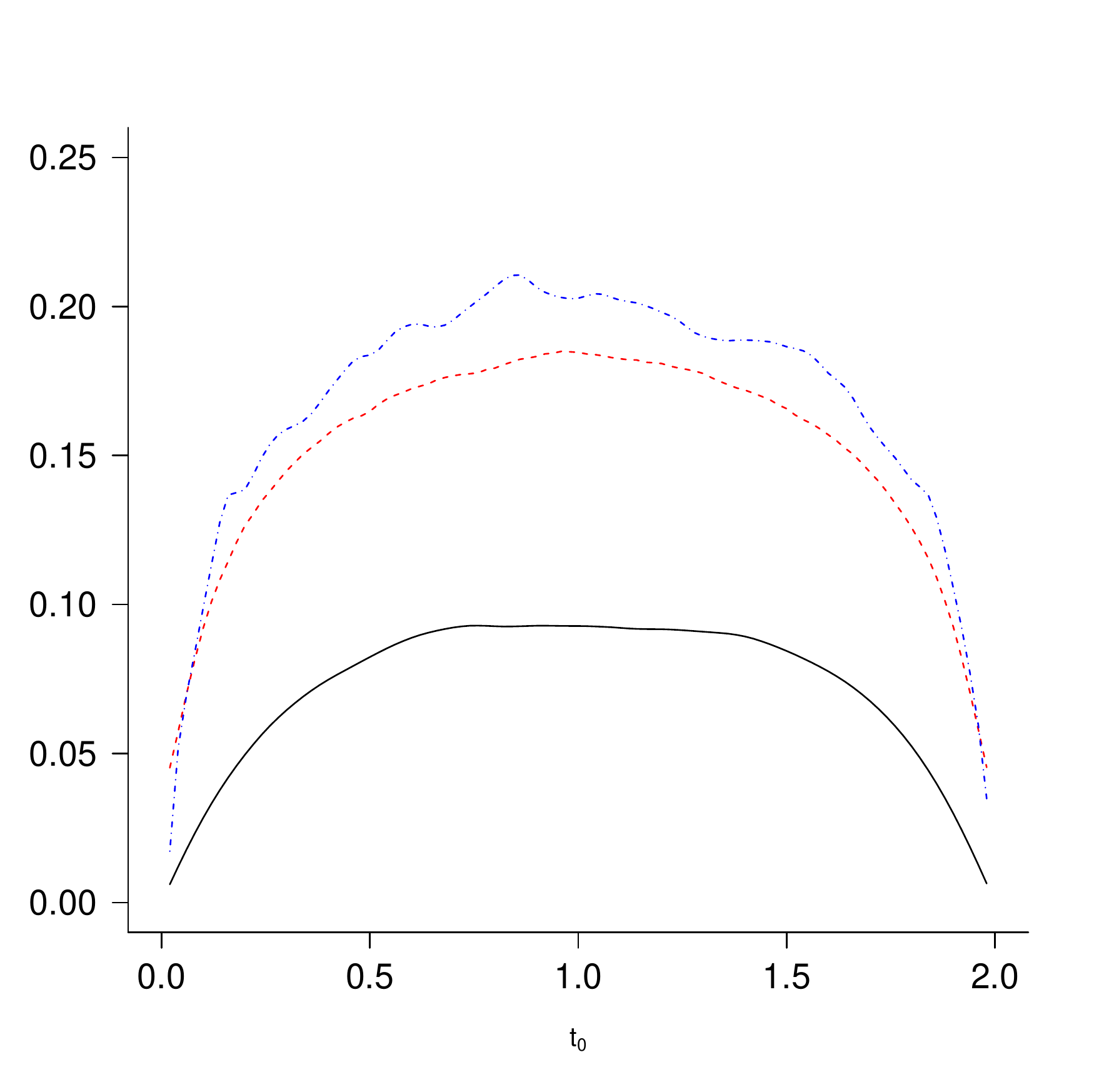}
		\caption{}
	\end{subfigure}	\\
	\caption{Uniform samples: Proportion of times that $F_0(t_i),\, t_i=0.02,0.04,\dots$ is not in the $95\%$ CI's in $5000$ samples using the Epanechnikov kernel and $1000$ bootstrap samples for (a) classical SMLE-based CIs (\ref{CI_type1}) (blue, dashed) and Studentized SMLE-based CIs (\ref{CI_type2}) (black, solid), (b) Banerjee-Wellner CIs (blue, dashed) and Studentized SMLE-based CIs (\ref{CI_type2}) (black, solid) and (c) Sen-Xu CIs (blue, dashed) and Studentized SMLE-based CIs (\ref{CI_type2}) (black, solid). (d) The average length for the SMLE-based CIs (\ref{CI_type2}) (black, solid), Banerjee-Wellner CIs (red, dashed) and Sen-Xu CIs (blue, dashed-dotted). $n=1000$ and $h=2n^{-1/5}$.}
	\label{fig:uniform1}
\end{figure}

In contrast to the MLE-based intervals, the SMLE-based intervals in the exponential setting are subjected to bias effects. A picture of the asymptotic bias $\b$ defined in (\ref{mu-sigma}) is shown in Figure \ref{fig:bias}, the function $\b = \b(t)$ is scaled by a factor $1000^{-2/5}$ and therefore its magnitude corresponds to the quantity that should be subtracted from the estimated SMLE-based CIs in order to construct unbiased confidence intervals based on $n=1000$ observations. Accurate procedures to handle the bias are hard to obtain and still need more investigation in further research. We propose to use a combination of a local bandwidth, minimizing an estimate of the Mean Squared Error together with undersmoothing in order to reduce the bias effects when constructing confidence intervals around the SMLE. Undersmoothing can be used to correct for bias when the bootstrap is used to construct confidence intervals. As argued by \cite{hall:92}, undersmoothing has the advantage that direct estimation of the bias is no longer necessary and can improve coverage accuracy of the CIs as well as result in narrower intervals. An improvement of the performance of bootstrap-based CIs around the SMLE as a consequence of undersmoothing is also observed in \cite{piet_geurt:14}, Section 9.5. (see e.g. Figure 9.19 on p.272).
\begin{figure}[!ht]
	\centering
	\includegraphics[width=0.4\textwidth]{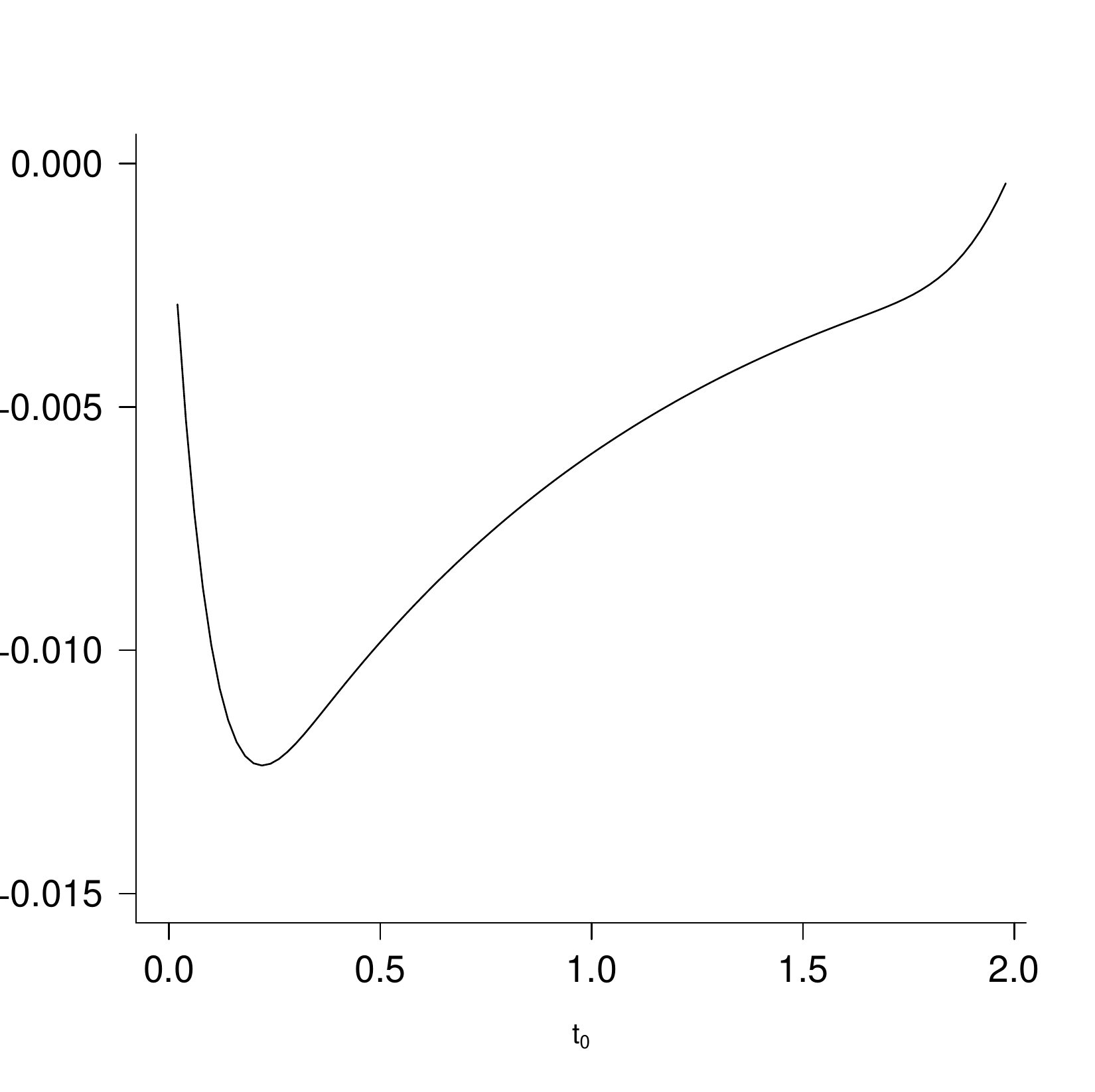}
	\caption{Truncated exponential samples: Bias for a sample of size $n=1000$, $h=2n^{-1/5}$. }
	\label{fig:bias}
\end{figure}

\subsection{Bandwidth selection}
We use a bootstrap procedure to select the optimal local bandwidth at time point $t$. The selection criterion is based on minimizing the Mean Squared Error (MSE)
\begin{align}
\label{mse}
{\rm MSE}(h) = E\{\tilde F_{nh}(t)-F_0(t)\}^2.
\end{align}
Since $F_0$ is unknown in practice, we select, for each time point $t$, the constant $\hat c_{t,opt}$ which minimizes 
\begin{align}
\label{mse-hat}
\widehat{{\rm MSE}}(c) =B^{-1}\sum_{b=1}^{B}\{\tilde F_{m,cm^{-1/5}}^{b}(t) - \tilde F_{n\tilde h_0}(t) \}^2
\end{align}
where $\tilde F_{m,cm^{-1/5}}^{b}$ is the SMLE in a bootstrap sample $(T_1^*,\dd_1^*),\dots,(T_m^*,\dd_m^*)$ of size $m < n$, where the $T_i^*$ are sampled from a kernel estimator for the distribution function $G$ of the censoring variable $T$ and where the $\dd_i^*$ are sampled from a Bernoulli distribution with probability $\tilde F_{nh_0}(T_i^*)$. Here $\tilde F_{nh_0}$ denotes the SMLE in the original sample (of size $n$) using the bandwidth $h_0 = c_0n^{-1/5}$ for some constant $c_0$ and $B$ equals the number of bootstrap samples. A similar procedure to select the constant $c$ when interest is in point estimation of $F_0(t)$ is proposed in \cite{piet_geurt_birgit:10}. For each time point $t$, we next choose the bandwidth
\begin{align*}
\hat h_{t,opt} = \hat c_{t,opt} n^{-1/4},
\end{align*}
where we use undersmoothing to reduce the bias effect in constructing CIs for $F_0(t)$.

An important point is the fact that we have to use subsampling, i.e. bootstrapping with a smaller sample size, for estimating the right bandwidth in a reasonable fashion, as argued convincingly in \cite{hall:90}. In the present case, we took $m=100$. If one does not use subsampling, the bias/variance comparison is not done in the right way, whereas our present scheme, taking $m=100$ versus the original sample size $n=1000$, seemed to give a reasonable estimate of the MSE, as was borne out by a comparison with the real MSE.
We estimated $MSE(c)$ on a grid $c =$ 0.05, 0.10,$\ldots$, 5, for a sample of size $n=1000$ by a Monte Carlo experiment with $N = 1000$ simulation runs by
\begin{align}
\label{Monte_carlo_MSE}
\widetilde{\text{\rm MSE}}(c) =N^{-1}\sum_{j=1}^N ( \tilde F_{n,cn^{-1/5}}^{j}(t) - F_0(t))^2,
\end{align}
where $\tilde F_{n,cn^{-1/5}}^{j}(t)$ is the estimate of $F_0(t)$ in the $j$th simulation run, $j= 1,\ldots,N$.  Figure \ref{fig:6}(a) compares the values of $c$ minimizing the Monte-Carlo estimate of MSE (\ref{Monte_carlo_MSE}) with the values of $c$ minimizing the bootstrap MSE (\ref{mse-hat}) as a function of $t$, and illustrates that the bootstrap MSE is a good estimate of (\ref{mse}). 

Figure \ref{fig:6}(b) compares the proportion of times that $F_0(t_i),\, t_i=0.02,0.04,\dots$ is not in the $95\%$ Studentized SMLE-based CI's (\ref{CI_type2}) for the truncated exponential model when a fixed bandwidth $h=2n^{-1/5}$ is used with the proportion obtained when a local bandwidth is used. We use the bandwidth $h(t) = (0.3412 + 0.1280t)n^{-1/4}$ which corresponds to the least squares regression line through the points $(t, cn^{-1/4})$ where $c$ is the value minimizing (\ref{Monte_carlo_MSE}) at timepoint $t$. 
An improvement in the coverage probabilities of the CIs is seen at the left end (i.e. the region where the bias is most prominent), indicating that it is indeed possible to obtain good CIs if undersmoothing in combination with a local optimal bandwidth is considered. The coverage proportions for the MLE-based methods of \cite{banerjee_wellner:2005} and \cite{SenXu2015} (results not shown) are similar to the proportions obtained for the uniform samples. Under our regularity conditions, our SMLE-based CIs have a better behavior than the MLE-based intervals near the boundary of the intervals in terms of coverage proportions and in the middle of the interval in terms of the length of the intervals. 

\begin{figure}[!ht]
	\centering
	\begin{subfigure}[b]{0.35\textwidth}
		\includegraphics[width=\textwidth]{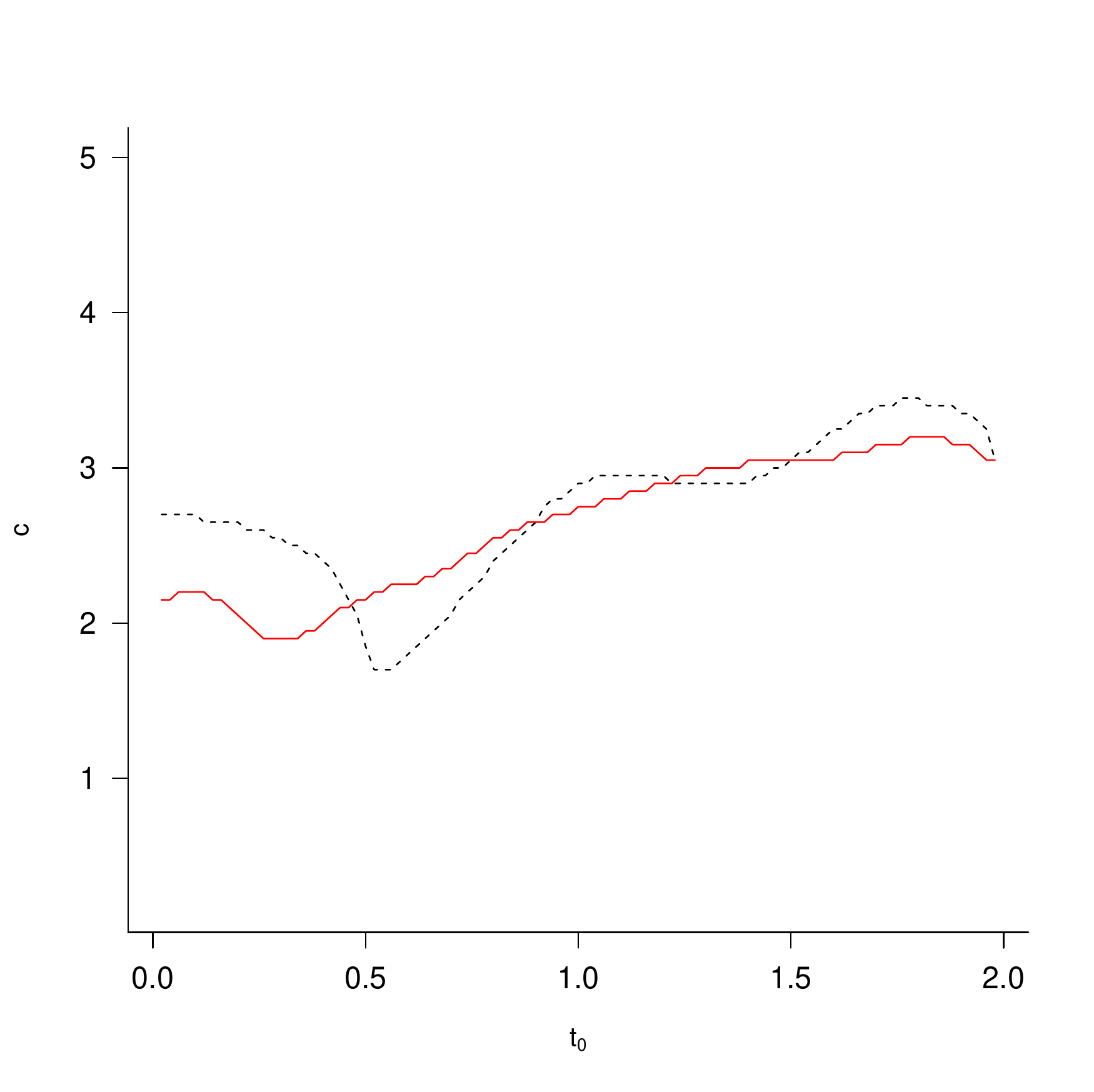}
		\caption{}
	\end{subfigure}
	\hspace{1cm}
	\begin{subfigure}[b]{0.35\textwidth}
		\includegraphics[width=\textwidth]{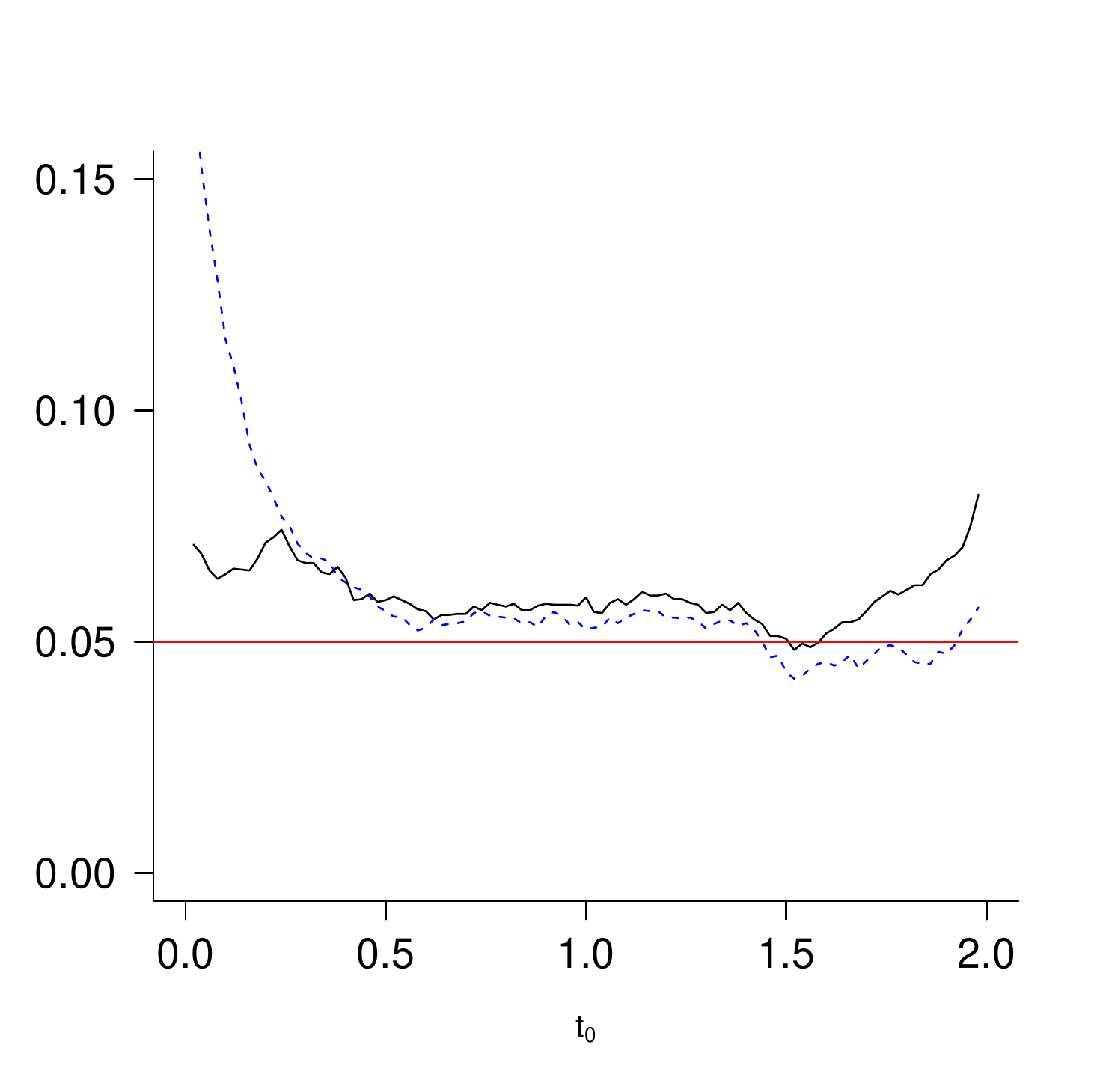}
		\caption{}
	\end{subfigure}	
	\caption{Truncated exponential samples: (a) optimal bandwidth $h$ at timepoint $t_0=0.02,0.04,\ldots,2$ obtained by the minimizer of $\widetilde{{\rm MSE}}$ (red, solid) using $N=1,000$ Monte-Carlo runs and $\widehat{{\rm MSE}}$ (black, dashed) using $B=1000$ bootstrap runs of size $m=100$ and $h_0=2n^{-1/5}$. (b)  Proportion of times that $F_0(t_i),\, t_i=0.02,0.04,\dots$ is not in the $95\%$ Studentized SMLE-based CI's (\ref{CI_type2}) in $5000$ samples using $1000$ bootstrap samples with bandwidth $h=2n^{-1/5}$ (blue, dashed) and $h(t) = (0.3412 + 0.1280t)n^{-1/4}$ (black, solid); $n=1000$.}
	\label{fig:6}
\end{figure}

The CIs for one sample  of size $n=1000$ are shown in Figure \ref{fig:uniform_ci} and Figure \ref{fig:exponential_ci}. Note that the Sen-Xu CIs do not have monotone bounds. One may wonder if one really wants to use the MLE for estimating the distribution function, if one resamples from the SMLE as in \cite{SenXu2015} since one uses smoothness conditions that allow to estimate the distribution function at a faster rate than the convergence rate of the MLE. The pointwise CIs around the SMLE change smoothly over the interval whereas MLE-based intervals change in discrete steps.

\begin{figure}[!ht]
	\centering
	\begin{subfigure}[b]{0.3\textwidth}
		\includegraphics[width=\textwidth]{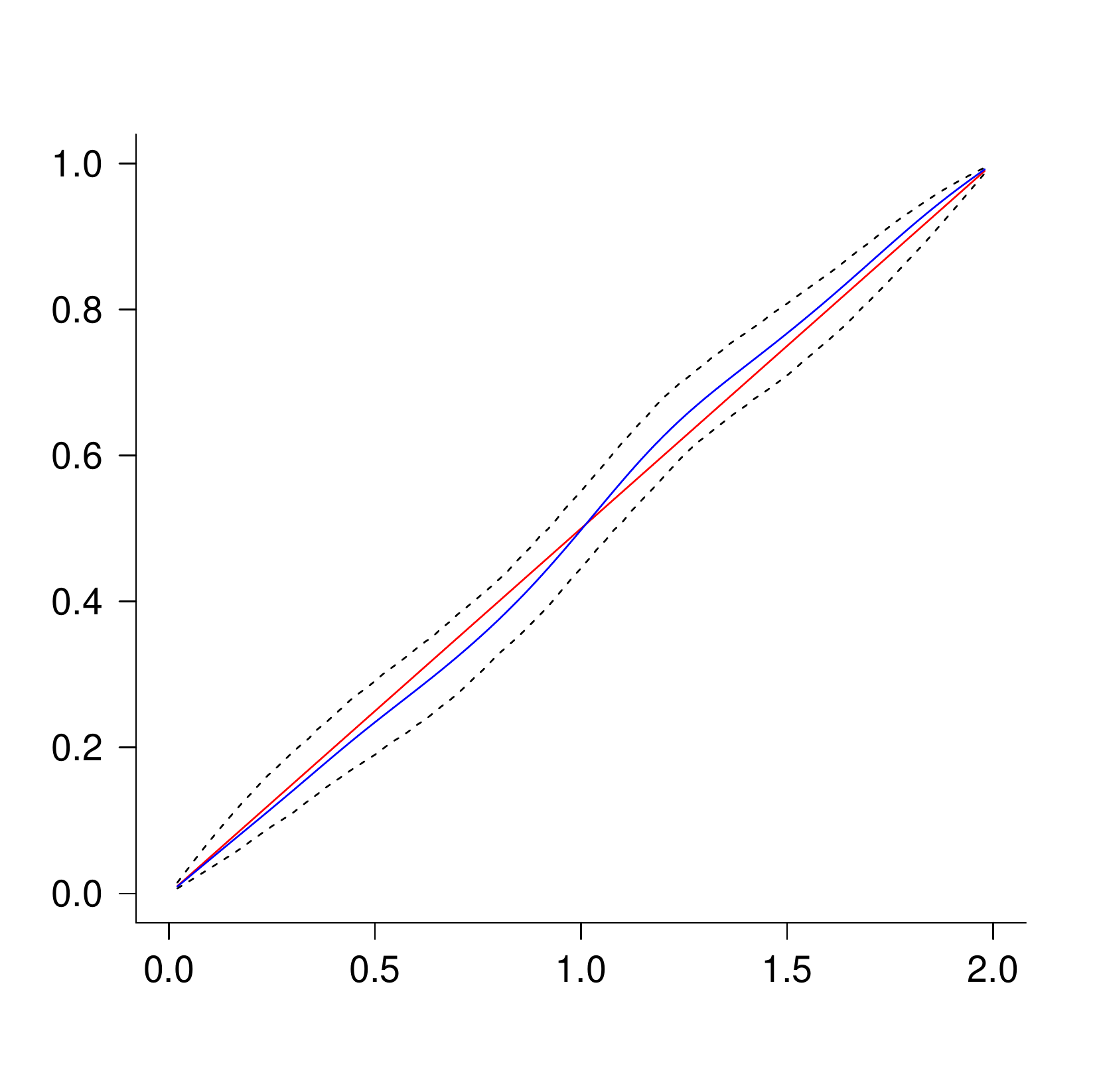}
		\caption{}
	\end{subfigure}
	\begin{subfigure}[b]{0.3\textwidth}
		\includegraphics[width=\textwidth]{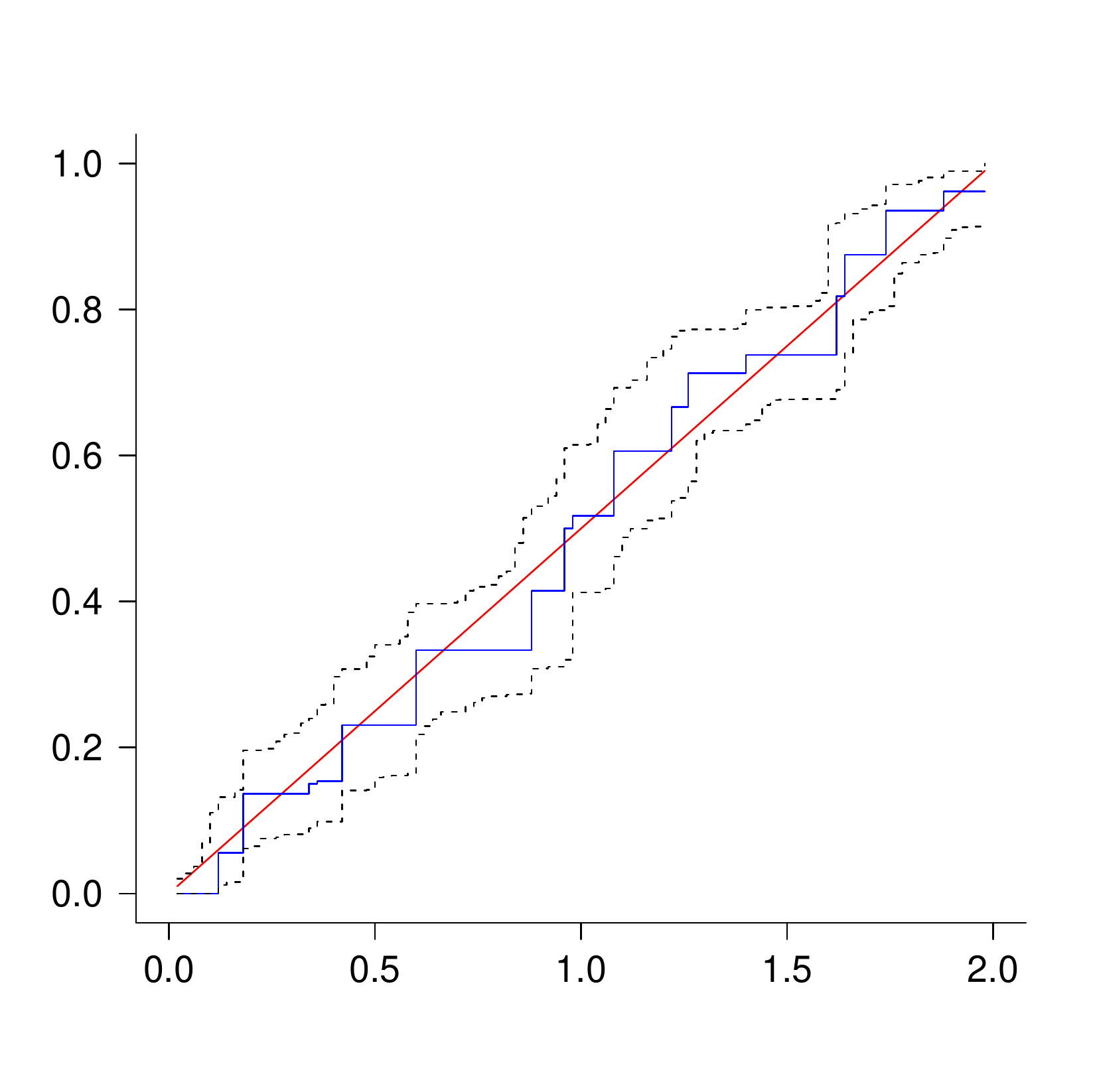}
		\caption{}
	\end{subfigure}
	\begin{subfigure}[b]{0.3\textwidth}
		\includegraphics[width=\textwidth]{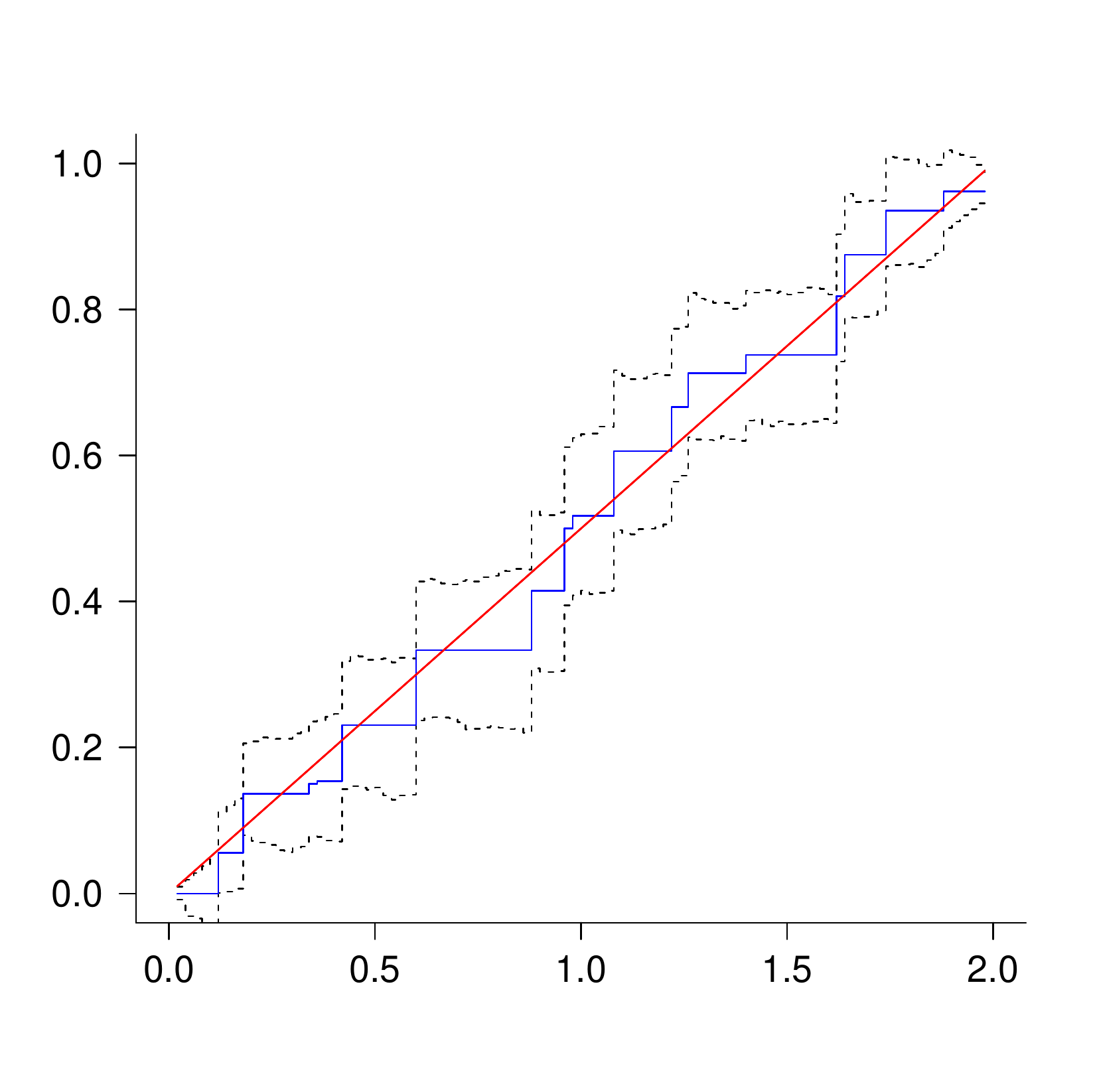}
		\caption{}
	\end{subfigure}
	\caption{Uniform samples: $F_0$ (red solid). (a) Studentized SMLE-based CI (\ref{CI_type2}), (b) Banerjee-Wellner CI and (c) Sen-Xu CI based on one sample of size $n=1000$ using $1000$ bootstrap samples. In (a) the SMLE (blue, solid) is given and in (b,c) the MLE (blue, step function) is given; $h=2n^{-1/5}$.  }
	\label{fig:uniform_ci}
\end{figure}

\begin{figure}[!ht]
	\centering
	\begin{subfigure}[b]{0.3\textwidth}
		\includegraphics[width=\textwidth]{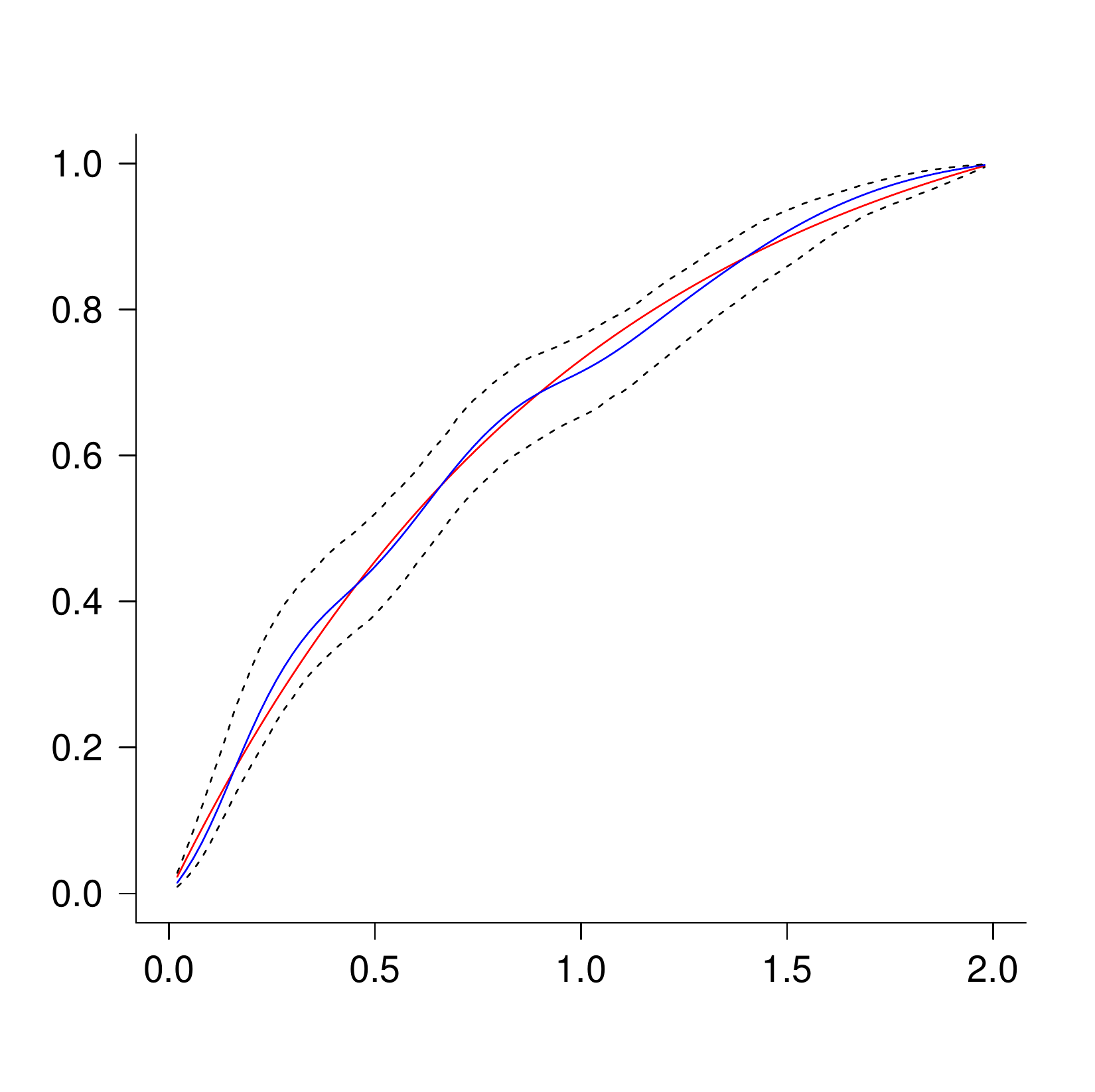}
		\caption{}
	\end{subfigure}
	\begin{subfigure}[b]{0.3\textwidth}
		\includegraphics[width=\textwidth]{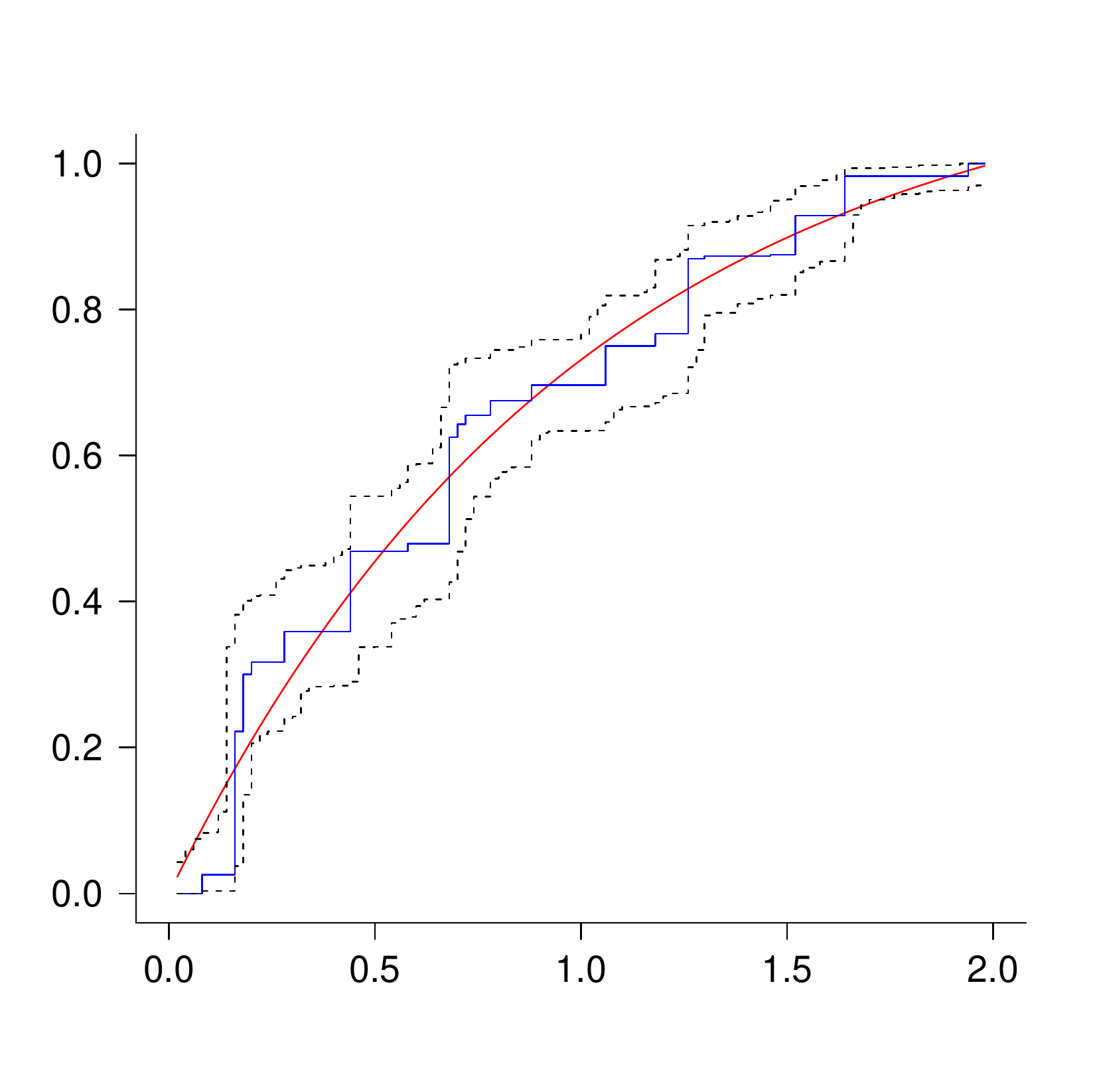}
		\caption{}
	\end{subfigure}
	\begin{subfigure}[b]{0.3\textwidth}
		\includegraphics[width=\textwidth]{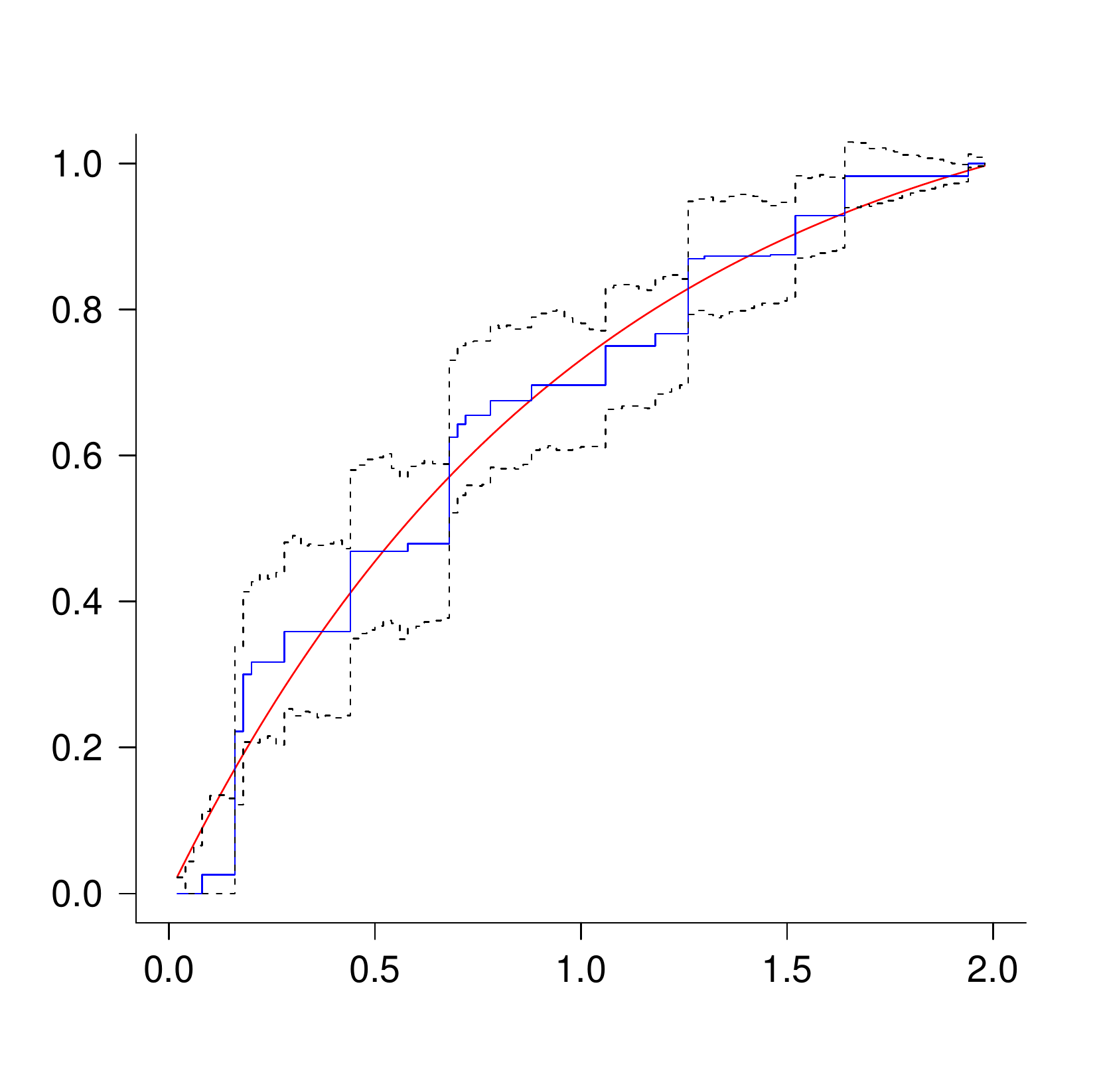}
		\caption{}
	\end{subfigure}
	\caption{Truncated exponential samples: $F_0$ (red solid). (a) Studentized SMLE-based CI (\ref{CI_type2}), (b) Banerjee-Wellner CI and (c) Sen-Xu CI based on one sample of size $n=1000$ using $1000$ bootstrap samples. In (a) the SMLE (blue, solid) is given and in (b,c) the MLE (blue, step function) is given. $h(t)=(0.3412 + 0.1280t)n^{-1/4}$ for SMLE-based CI and $h=2n^{-1/5}$ for Sen-Xu CI.  }
	\label{fig:exponential_ci}
\end{figure}

\section{Real data analysis}
\label{section:realdata}
\subsection{Hepatitis A}
\cite{keiding:91} considered a cross-sectional study on the Hepatitis A virus from Bulgaria. In 1964 samples were collected from schoolchildren and blood donors on the presence or absence of Hepatitis A immunity. In total $n=850$ individuals ranging from 1 to 86 years old were tested for immunization. It is assumed that, once infected with Hepatitis A, lifelong immunity is achieved. We are interested in estimating the sero-prevalence for Hepatitis A in Bulgaria.  We constructed confidence intervals at timepoints $t_1 = M/100,t_2=2M/100,\ldots, M$ where $M=86$ is the largest observed age using the Studentized SMLE-based CIs (\ref{CI_type2}) described in Section \ref{section:pointwiseCI} using a local bandwidth $h(t_i) = (0.5M + 1.5t_i)n^{-1/5}$. A picture of the CIs together with the likelihood-ratio based CIs of \cite{banerjee_wellner:2005} and the CIs of \cite{SenXu2015} is given in Figure \ref{fig:hepA_ci}. The estimated prevalence of Hepatitis A at the age of 18 is 0.51, about half of the infections in Bulgaria happen during childhood. The length of the CIs is smallest for our SMLE-based CIs and largest for the Sen-Xu CIs. The latter CIs have left and right end points that are not monotone increasing in age, a property that is not shared by the other two CIs which have monotone increasing bounds. In contrast to the Banerjee-Wellner CIs, the bounds of our SMLE based CIs are not increasing by construction.

\begin{figure}[!ht]
	\centering
	\begin{subfigure}[b]{0.3\textwidth}
		\includegraphics[width=\textwidth]{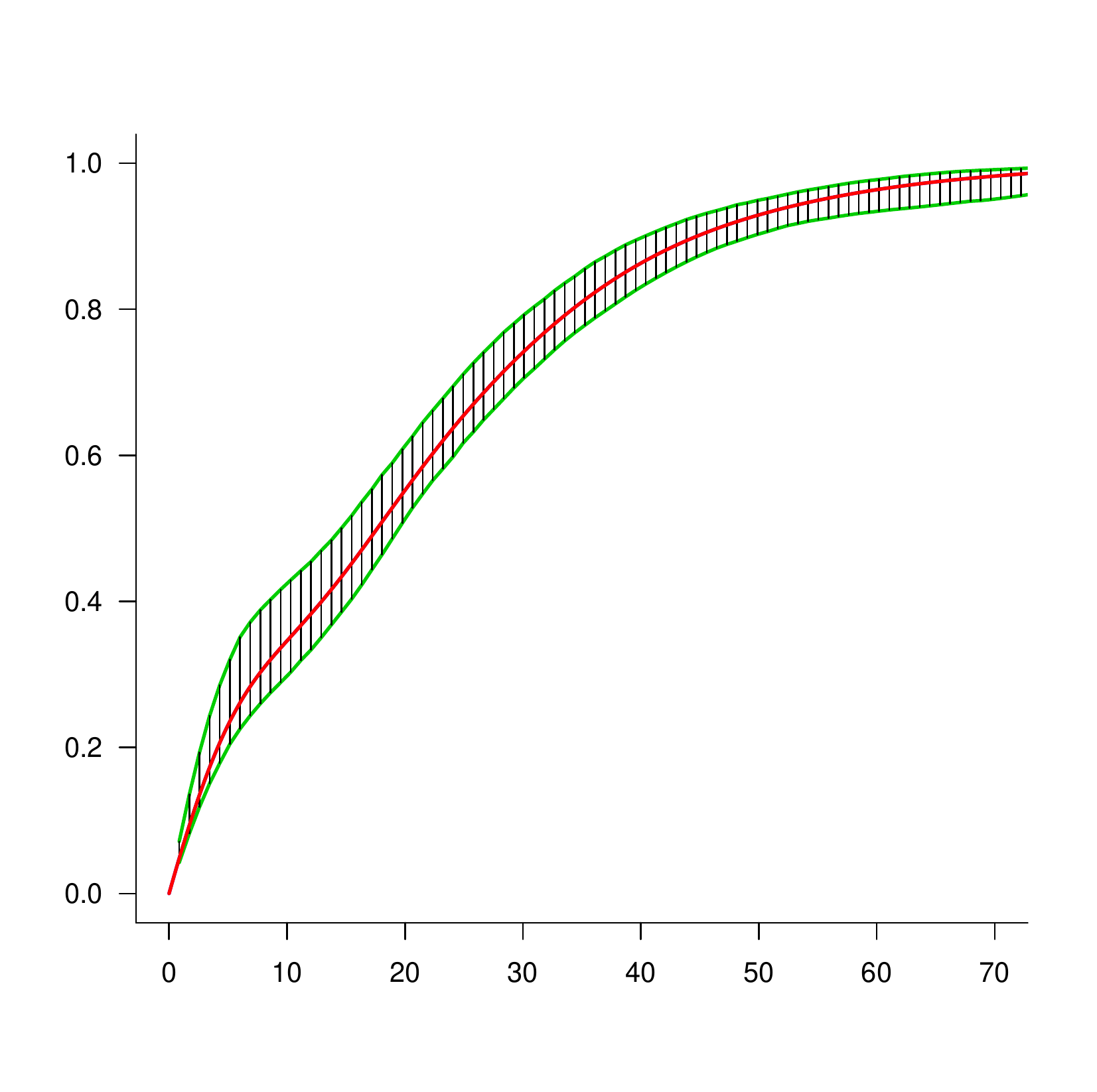}
		\caption{}
	\end{subfigure}
	\begin{subfigure}[b]{0.3\textwidth}
		\includegraphics[width=\textwidth]{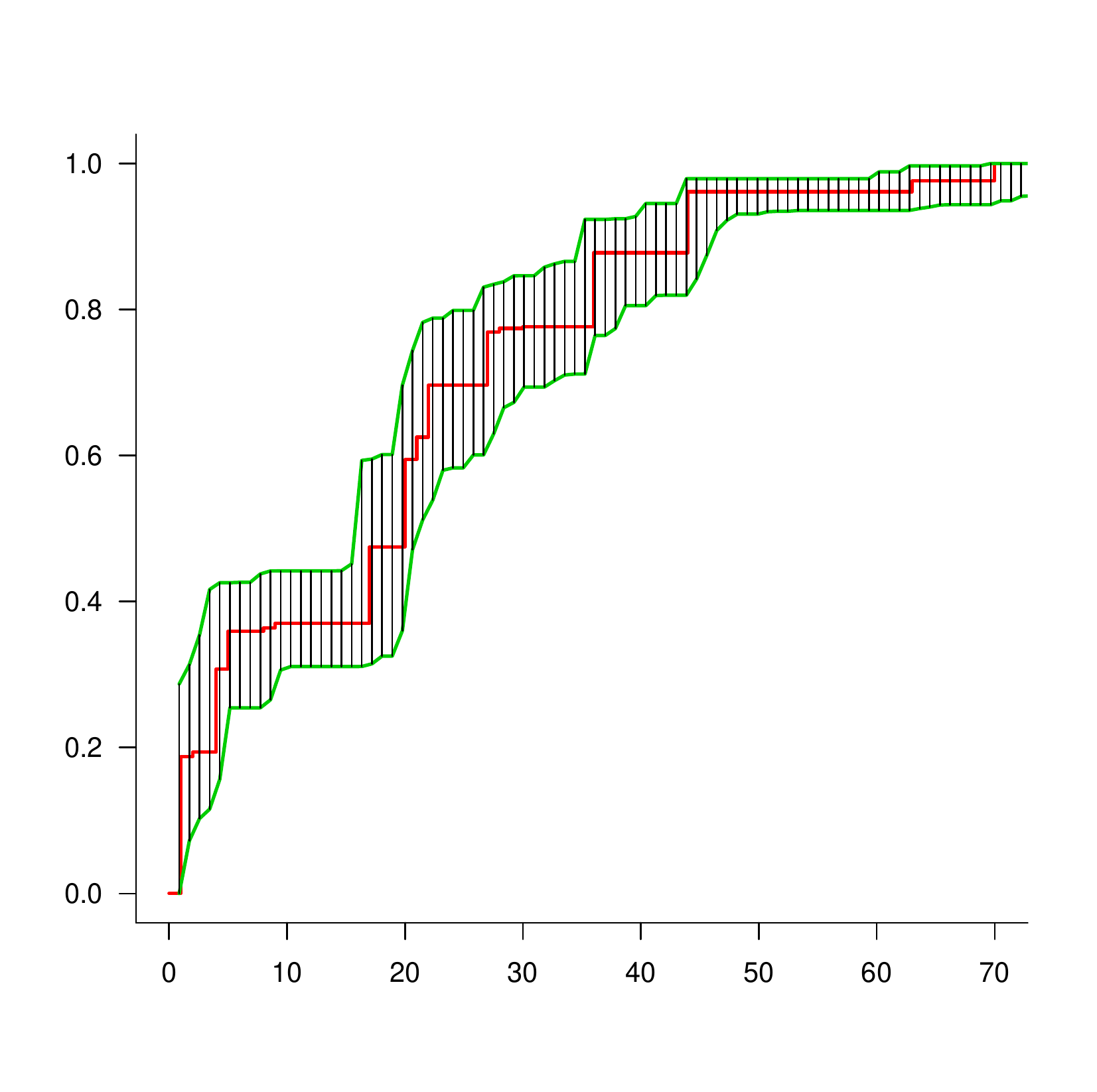}
		\caption{}
	\end{subfigure}
	\begin{subfigure}[b]{0.3\textwidth}
		\includegraphics[width=\textwidth]{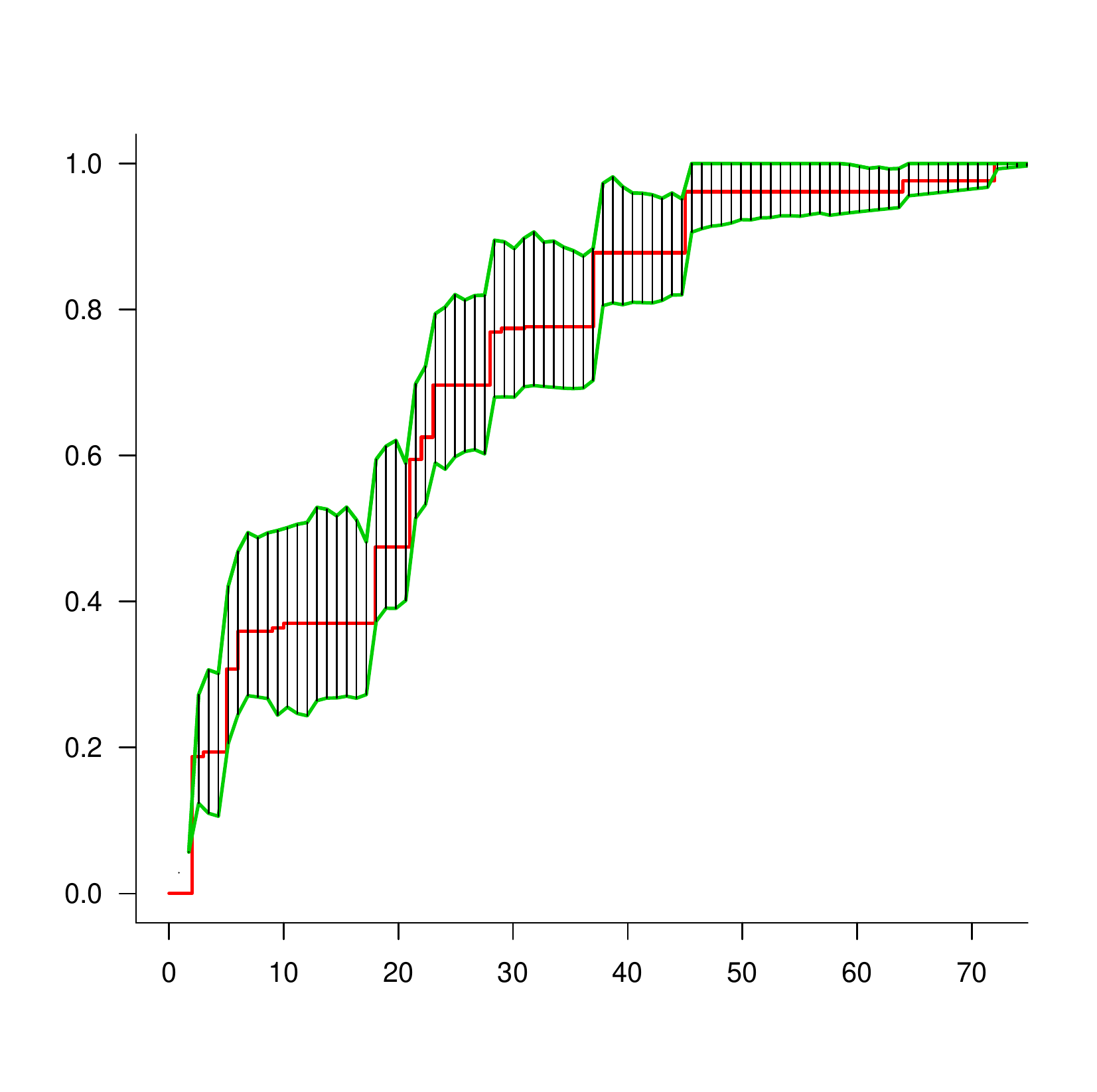}
		\caption{}
	\end{subfigure}
	\caption{Hepatitis A data: (a) Studentized SMLE-based CI (\ref{CI_type2}), (b) Banerjee-Wellner CI and (c) Sen-Xu CI based on $n=850$ observations using $1000$ bootstrap samples. In (a) the SMLE (red, solid) is given and in (b,c) the MLE (red, step function) is given. $h(t)=(43+1.5t)n^{-1/5}$ for SMLE-based CI and $h=86n^{-1/5}$ for Sen-Xu CI.  }
	\label{fig:hepA_ci}
\end{figure}

\subsection{Rubella}
\cite{Keiding:96} considered a current status data set on the prevalence of rubella in 230 Austrian males older than three months. Rubella is a highly contagious childhood disease spread by airborne and droplet transmission. The symptoms (such as rash, sore throat, mild fever and swollen glands) are less severe in children than in adults. Since the Austrian vaccination policy against rubella only vaccinated girls, the male individuals included in the dataset represent an unvaccinated population and (lifelong) immunity could only be acquired if the individual got the disease. We are interested in estimating the time to immunization (i.e. the time to infection) against rubella using the SMLE. We constructed CIs  at timepoints $t_1 = M/100,t_2=2M/100,\ldots, M$ where $M=80.1178$ is the largest observed age, using CIs defined in (\ref{CI_type2}) with the boundary correction described in Section \ref{section:pointwiseCI} and a local bandwidth $h(t_i) = (0.25M + t_i)n^{-1/5}$ if $t_i\leq 20$ and $h(t_i) =h(20) + 2(t_{i}-20)n^{-1/5}$ else. The bandwidth choice is based on the fact that most infections occurred before the age of 20 years and a larger bandwidth is needed in the range [20,M] to obtain plausible estimates. The SMLE increases steeply in the ages before adulthood which is in line with the fact that rubella is considered as a childhood disease. As can be seen from Figure \ref{fig:Rubella_ci}, our CIs and the Banerjee-Wellner CIs are favored over the Sen-Xu intervals due to their remarkable non-increasing behavior and their large width in the region up to 20 years. A further discussion of statistical aspects of this data set can be found in \cite{banerjee_wellner:2005} and \cite{piet_geurt:14}.

\begin{figure}[!ht]
	\centering
	\begin{subfigure}[b]{0.3\textwidth}
		\includegraphics[width=\textwidth]{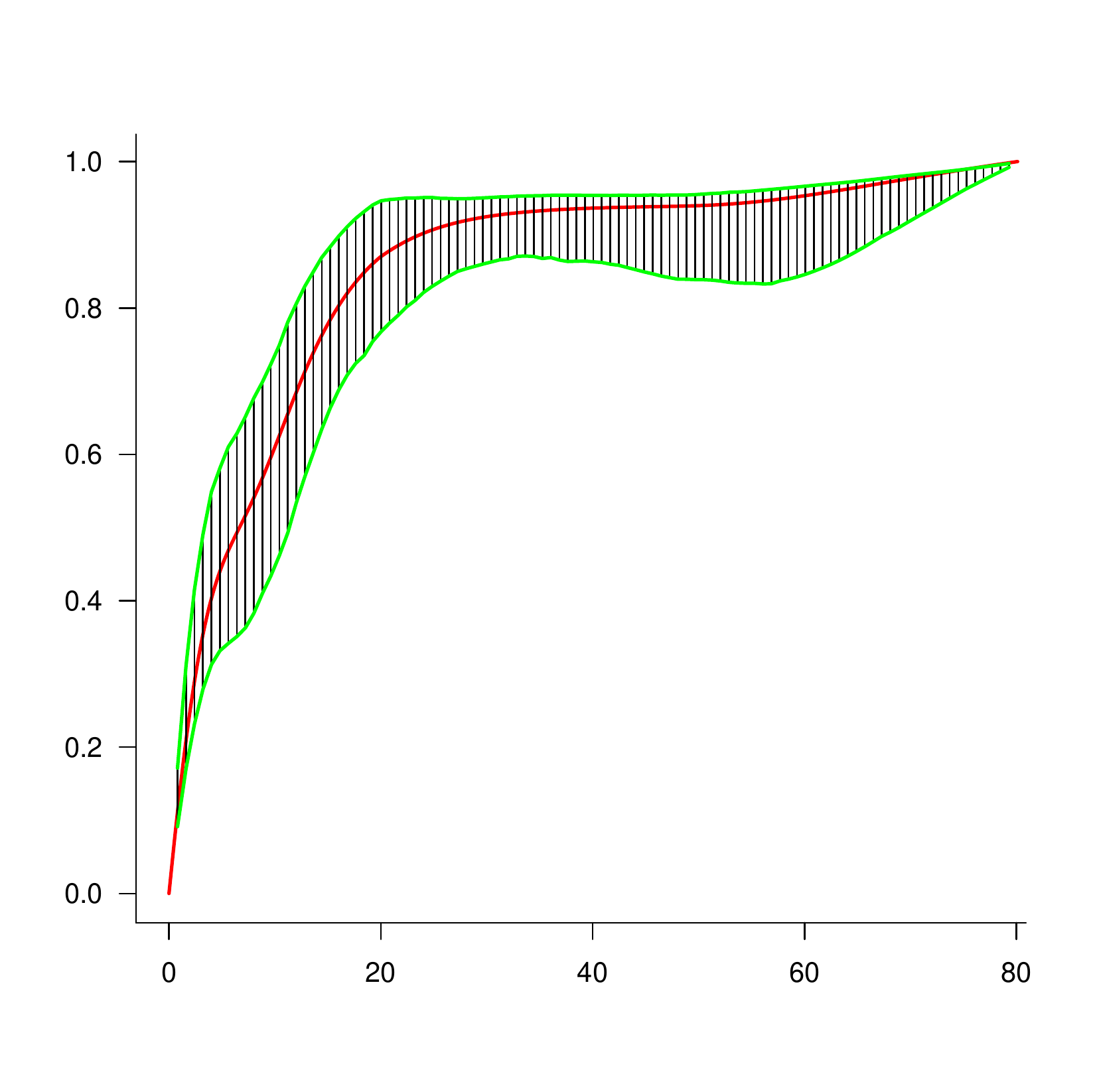}
		\caption{}
	\end{subfigure}
	\begin{subfigure}[b]{0.3\textwidth}
		\includegraphics[width=\textwidth]{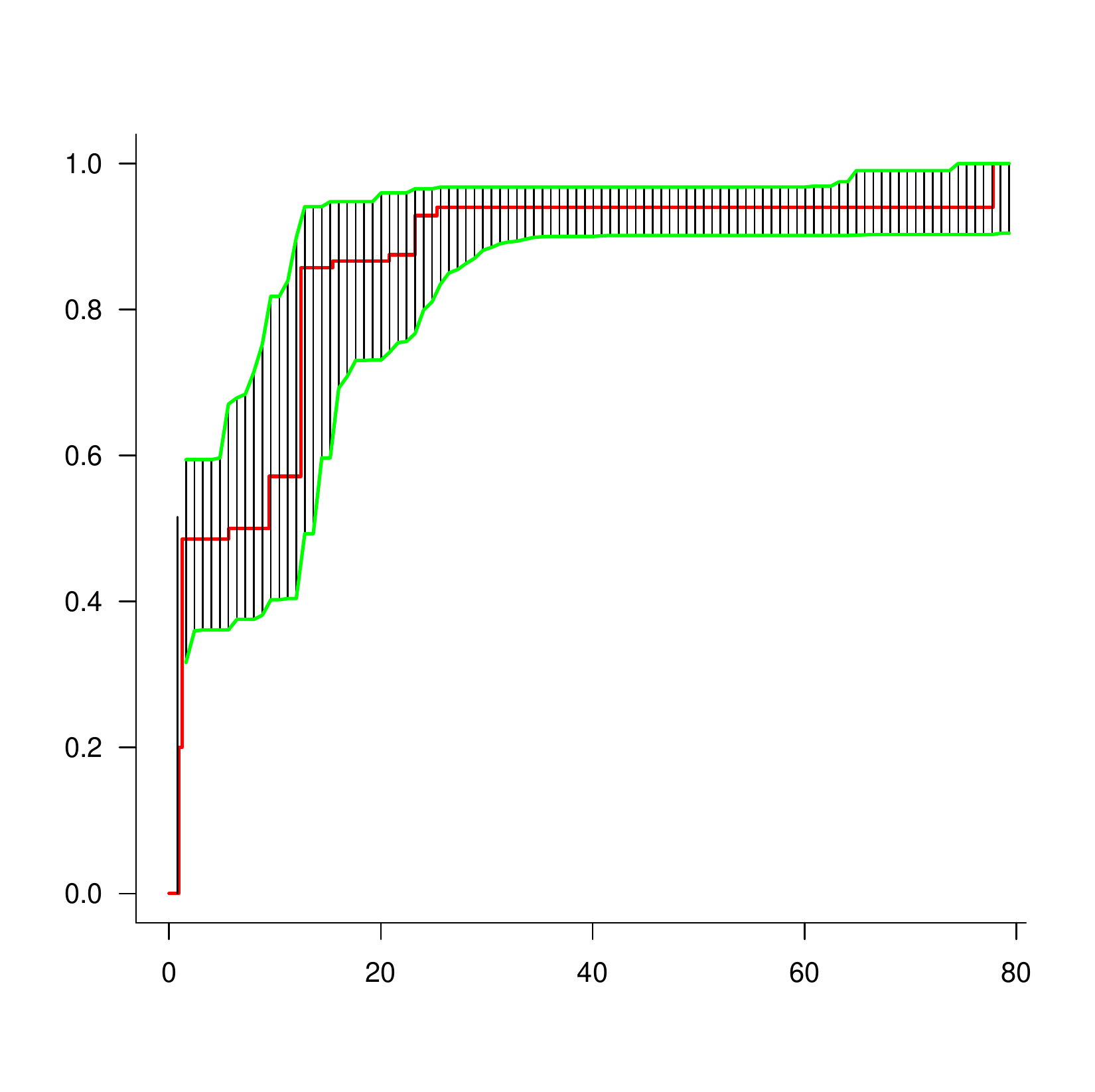}
		\caption{}
	\end{subfigure}
	\begin{subfigure}[b]{0.3\textwidth}
		\includegraphics[width=\textwidth]{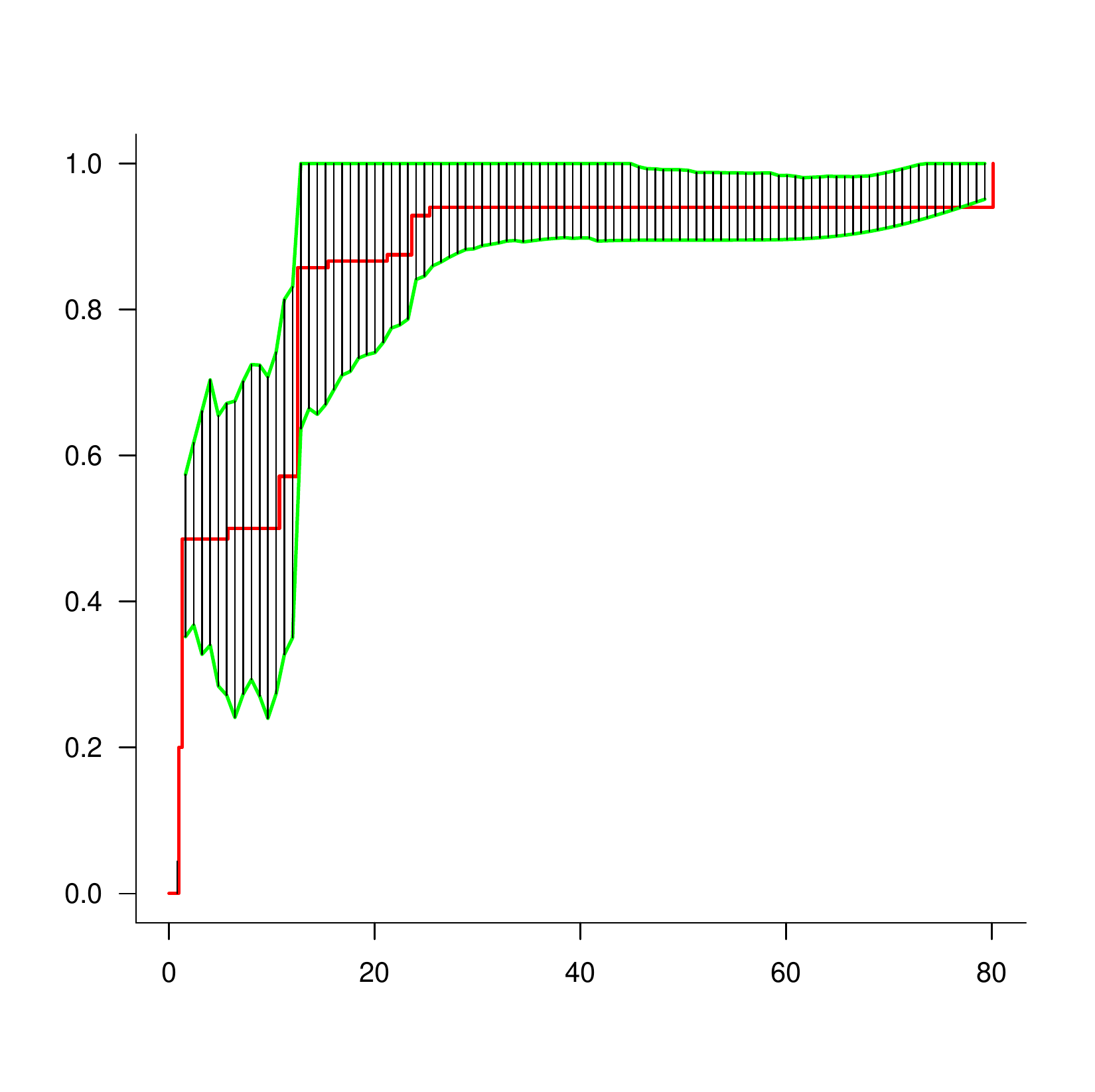}
		\caption{}
	\end{subfigure}
	\caption{Rubella data: (a) Studentized SMLE-based CI (\ref{CI_type2}), (b) Banerjee-Wellner CI and (c) Sen-Xu CI based on  $n=230$ observations using $1000$ bootstrap samples. In (a) the SMLE (red, solid) is given and in (b,c) the MLE (red, step function) is given. $h(t) = (0.25M + t)n^{-1/5}$ if $t\leq 20$ and $h(t) =h(20) + 2(t-20)n^{-1/5}$ else for SMLE-based CI and $h=80.12n^{-1/5}$ for Sen-Xu CI.  }
	\label{fig:Rubella_ci}
\end{figure}

\section{Concluding remarks}
\label{section:concluding remarks}
In this paper we presented a method for confidence interval estimation for the distribution function of a random variable which cannot be observed completely due to current status censoring. The CIs are based on a smooth bootstrap procedure. 
Unfortunately,  a rather negative feeling on the usefulness of bootstrap methods in this context is created by the results in \cite{abrevaya_huang2005} and \cite{kosorok:08}, showing that the classical bootstrap can not be used in reproducing the ``Chernoffian" limit distribution of the MLE in current status models and of the Grenander estimator in monotone density estimation. The result in \cite{sen_mouli_woodroofe:10} showing that even resampling from the Grenander estimator itself will not result in a consistent bootstrap has further contributed to this negative image of the bootstrap. 

A positive bootstrap result, on the other hand was derived in \cite{SenXu2015} showing that one can in fact reproduce the Chernoffian limit distribution if one resamples from a smooth estimate of the distribution function, such as the smoothed maximum likelihood estimator (SMLE). But we meet a familiar paradox in the field here: if one introduces smoothness conditions (which is also done in the conditions of limit theorems for the MLE), then one can usually achieve better convergence rates than the MLE achieves. For example, under the smoothness conditions of \cite{piet_geurt_birgit:10}, the SMLE achieves rate $n^{2/5}$ (familiar from density estimation), whereas the MLE only achieves rate $n^{1/3}$ (familiar from histogram estimation). The authors of \cite{SenXu2015} however use bootstrapping by resampling the indicators $\dd_i^*$ from the SMLE, while keeping the observation times $T_i$ fixed in combination with intervals around the MLE instead of the smooth estimate from which the resampling is done. It seems more natural to construct  confidence intervals on the basis of the SMLE instead of the MLE and this is indeed what we propose in the current paper. We have shown that the procedure, based on the SMLE, gives a consistent bootstrap, and has considerably smaller intervals than the intervals in \cite{banerjee_wellner:2005}, who used LR tests, based on the  (restricted) MLE, or \cite{SenXu2015} who used intervals, based on the MLE rather than the SMLE.

We showed in this article that the intervals based on the SMLE can be constructed in such a way that one gets a better boundary behavior, provided the necessary smoothness conditions are satisfied. The simulations also showed, not unexpectedly,  that the Studentized CIs were  better than the non-Studentized bootstrap CIs. In contrast to the unbiased MLE, the squared bias and variance for the SMLE are of the same order. We therefore found in our simulations that the performance of our CIs increases considerably if we subtracted the (unobserved) bias in the construction of the CIs. In practice it is of course not possible to subtract the real bias. However, our simulations showed a remarkable improvement of the behavior of the CIs if one uses a local bandwidth in combination with undersmoothing instead of one global bandwidth of order $n^{-2/5}$. We propose a bandwidth selection criteria based on the smooth bootstrap procedure developed in this paper and apply the concept of undersmoothing to reduce the bias effect when constructing confidence intervals around the SMLE. Further research related to the development of criteria to decide on how to adapt the bandwidth in order to handle the bias are worth studying in further research.

Rcpp scripts for producing the pictures of this paper and doing simulations can be found in \cite{github:15}.

\section*{Acknowledgements}
\label{section:acknowledgements}
We are very grateful to C\'ecile Durot for communicating her approach
to the proof of Lemma \ref{lemma_Th11.3} to us. The research of the second author was supported by the Research Foundation Flanders (FWO) [grant number 11W7315N]. Support from the IAP Research Network P7/06 of the Belgian State (Belgian Science Policy) is gratefully acknowledged. For the simulations we used the infrastructure of the VSC - Flemish Supercomputer Center, funded by the Hercules Foundation and the Flemish Government - department EWI.

\section{Appendix}
\label{section:appendix}
\subsection{Proof of Theorem \ref{th:bootstrap_SMLE}}
We denote the bootstrap sample by $(T_1,\dd_1^*),\dots,(T_n,\dd_n^*)$. Note that the sample is produced by keeping the $T_i$ fixed and drawing the $\dd_i^*$ from a Bernoulli distribution with probability $\tilde F_{nh}(T_i)$ at each $i$th draw.  
Let $\G_n$ be the empirical measure of $T_1,\ldots, T_n$ and let $\P_n^*$ denote the empirical measure of $(T_1,\dd_1^*),\dots,(T_n,\dd_n^*)$. We write
\begin{align*}
n^{-1} \sum_{i=1}^n f(T_i,\dd_i^*) = \int f(u,\d^*)\,d\P_n^*(u,\d^*).
\end{align*}
for some bounded function  $f:[0,M]\times\{0,1\}\to\R$. Note that for any bounded function $h:[0,M]\to\R$
\begin{align*}
n^{-1} \sum_{i=1}^n h(T_i) = \int h(u)\,d\P_n^*(u,\d^*) = \int h(u)\,d\G_n(u).
\end{align*}
Finally let  $P_n^*$ denote the conditional probability measure, given $(T_1,\dd_1),\dots,(T_n,\dd_n)$ and note that
\begin{align}
\label{expectation-delta}
P_n^*\left( \dd_i^* = 1\right) =\tilde F_{nh}(T_i) \quad i=1,\ldots,n.
\end{align}
For the proof of Theorem \ref{th:bootstrap_SMLE} we use the so-called ``switch relation'', which reduces the study of $\hat F_n^*$ to the study of an inverse process. To this end, we define the process $W_n^*$ by:
\begin{align*}
W_n^*(t)=n^{-1}\sum_{i=1}^n\dd_i^*1_{\{T_i\le t\}}.
\end{align*}
and the process (in $a$) $U_n^*$ by:
\begin{equation}
\label{argmin_process}
U_n^*(a)=\argmin\{t\in\R:W_n^*(t)-a\G_n(t)\}.
\end{equation}
Then, taking $a_n=\tilde F_{nh}(t)$, we get the {\it switch relation}:
\begin{align}
\label{switch_curstat}
P_n^*\left\{n^{1/3}\{\hat F_n^*(t)-\tilde F_{nh}(t)\}\ge x\right\}=
P_n^*\left\{\hat F_n^*(t)\ge a_n+n^{-1/3}x\right\}=P_n^*\left\{U_n^*(a_n+n^{-1/3}x)\le t\right\},
\end{align}

Now, let $U_n$ be defined by
\begin{align}
\label{def_U_n}
U_n(a)=\inf\{x\in\R:\tilde F_{nh}(x)\ge a\},\qquad a\in(0,1).
\end{align}
We have the following result.
\begin{lemma}
	\label{lemma_Th11.3}
	There are positive constants $C_1$ and $C_2$, such that, almost surely, for all $x>0$ and all large $n$:
	\begin{align*}
	P_n^*\left\{n^{1/3}\left|U_n^*(a)-U_n(a)\right|\ge x\right\}\le C_1e^{-C_2 x^3}.
	\end{align*}
\end{lemma}
Note that, in the unconditional setting, this is Theorem 11.3 in \cite{piet_geurt:14}.  Let  $E_n^*$ denote the conditional expectation, given $(T_1,\dd_1),\dots,(T_n,\dd_n)$, it follows from Lemma \ref{lemma_Th11.3} and the switch-relation that
\begin{align}
\label{L2_bound_expectation}
E_n^*\left\{\hat F_n^*(t)-\tilde F_{nh}(t)\right\}^2 \leq Kn^{-2/3} \quad \forall t \in [0,M],
\end{align}
which moreover implies that
\begin{align}
\label{bootstrap_L_2-bound}
\|\hat F_n^*-\tilde F_{nh}\|_2=O_p^*\left(n^{-1/3}\right),
\end{align}
where $O_p^*\left(n^{-1/3}\right)$ means that for all $\e>0$ and almost all sequences $(T_1,\dd_1),(T_2,\dd_2),\dots$, there exists an $M>0$ such that
\begin{align*}
P_n^*\left\{n^{1/3}\|\hat F_n^*-\tilde F_{nh}\|_2\ge M\right\}<\e,
\end{align*}
for all large $n$.
We also have, similarly:
\begin{equation}
\label{local_L_2_bound}
\int_{t-h}^{t+h}\bigl\{\hat F_n^*(x)-\tilde F_{nh}(x)\bigr\}^2\,dx= O_p^*\left(h n^{-2/3}\right),
\end{equation}
conditionally on $(T_1,\dd_1),(T_2,\dd_2),\dots$. See p.\ 320 of \cite{piet_geurt:14} for the relation of Lemma \ref{lemma_Th11.3} to these last statements. We now give the proof of Theorem \ref{th:bootstrap_SMLE}, using the result of Lemma \ref{lemma_Th11.3}. The proof of Lemma \ref{lemma_Th11.3} is given at the end of this section.

\begin{proof}[Proof of Theorem \ref{th:bootstrap_SMLE}]
	Define the functions
	\begin{align*}
	\psi_{t,h}(u) = \frac{K_h(t-u)}{g(u)}
	\end{align*}
	and
	\begin{align*}
	\bar\psi_{t,h}^*(u)=
	\left\{\begin{array}{lll}
	\psi_{t,h}(\t_i),\,&\mbox{ if }\tilde F_{nh}(u)>\hat F_{n}^*(\t_i),\,u\in[\t_i,\t_{i+1}),\\
	\psi_{t,h}s),\,&\mbox{ if }\tilde F_{nh}(u)=\hat F_{n}^*(s),\mbox{ for some }s\in[\t_i,\t_{i+1}),\\
	\psi_{t,h}(\t_{i+1}),\,&\mbox{ if }\tilde F_{nh}(u)<\hat F_{n}^*(\t_i),\,u\in[\t_i,\t_{i+1}),
	\end{array}
	\right.
	\end{align*}
	where the $\t_i$ are the points of jump of $\hat F_{n}^*$. By the convex minorant interpretation of $\hat F_n^*$ we have,
	\begin{align*}
	&\int \bar\psi_{t,h}^*(u) \left\{\d^* - \hat F_{n}^*(u) \right\} d\P_n^*(u,\d^*)=0.
	\end{align*}
	This implies that,
	\begin{align*}
	0 &= \int \bar\psi_{t,h}^*(u) \left\{\d^* - \hat F_{n}^*(u) \right\}d\P_n^*(u,\d^*) \\
	&= \int \psi_{t,h}(u) \left\{\d^* - \hat F_{n}^*(u) \right\}d\P_n^*(u,\d^*)
	+\int \left\{\bar\psi_{t,h}^*(u) - \psi_{t,h}(u) \right \} \left\{\d^* - \hat F_{n}^*(u) \right\}d\P_n^*(u,\d^*)  \\
	&=  \int \psi_{t,h}(u) \left\{\d^* - \tilde F_{nh}^*(u) \right\}d(\P_n^*- P_n^*)(u,\d^*) \\
	& \qquad  + \int \psi_{t,h}(u) \left\{ \tilde F_{nh}^*(u)  - \hat F_{n}^*(u) \right\}d\P_n^*(u,\d^*)\\
	& \qquad  + \int \left\{\bar\psi_{t,h}^*(u) - \psi_{t,h}(u) \right \} \left\{\d^* - \hat F_{n}^*(u) \right\}d\P_n^*(u,\d^*),
	\end{align*}
	where we write $d(\P_n^*- P_n^*)$ instead of $d\P_n^*$ in the last equality as a result of (\ref{expectation-delta}).
	Using integrating by parts we have,
	\begin{align*}
	\tilde F_{nh}^*(t)-\int \IK_h(t-u)\,d\tilde F_{nh}(u)
	=\int \psi_{t,h}(u)\,\left\{\hat F_n^*(u)-\tilde F_{nh}(u)\right\}dG(u).
	\end{align*}
	So we find,
	\begin{align*}
	&\tilde F_{nh}^*(t)-\int \IK_h(t-u)\,d\tilde F_{nh}(u)\\
	&= \int \bar\psi_{t,h}^*(u) \left\{\d^* - \hat F_{n}^*(u) \right\}d\P_n^*(u,\d^*) -\int \psi_{t,h}(u)\,\left\{ \tilde F_{nh}(u)-\hat F_n^*(u)\right\}dG(u)\\
	&= \int \psi_{t,h}(u) \left\{\d^* - \tilde F_{nh}(u) \right\}d(\P_n^*- P_n^*)(u,\d^*) \\
	& \qquad  + \int \psi_{t,h}(u) \left\{ \tilde F_{nh}(u)  - \hat F_{n}^*(u) \right\}d(\G_n-G)(u,\d^*)\\
	& \qquad  + \int \left\{\bar\psi_{t,h}^*(u) - \psi_{t,h}(u) \right \} \left\{\d^* - \hat F_{n}^*(u) \right\}d\P_n^*(u,\d^*)
	\\& = A_{I} + A_{II} + A_{III}. 
	\end{align*}
	To study the asymptotic distribution of $$n^{2/5}\left\{\tilde F_{nh}^*(t)-\int \IK_h(t-u)\,d\tilde F_{nh}(u)\right\},$$
	we therefore have to analyze the three terms $A_{I},  A_{II}$ and $A_{III}$.	We start with $A_I$ and prove:
	\begin{align}
	\label{term1}
	n^{2/5}\int \psi_{t,h}(u) \left\{\d^* - \tilde F_{nh}(u) \right\}d(\P_n^*- P_n^*)(u,\d^*) \stackrel{{\cal D}}{\longrightarrow} N\left(0,\sigma^2\right)
	\end{align}
	where $\sigma^2$ is defined in (\ref{mu-sigma}).
	Define 
	$$Z_{nh,i} = n^{-3/5}\psi_{t,h}(T_i) \left\{\dd_i^* - \tilde F_{nh}(T_i) \right\} .$$
	The left hand side of (\ref{term1}) can then be expressed as $\sum_{i=1}^{n}Z_{nh,i}$. Conditionally on $(T_1,X_1),\ldots, (T_n, X_n)$, $Z_{nh,i}$ has mean zero and variance
	$$ \sigma_{nh,i}^2 = n^{-6/5}\psi_{t,h}^2(T_i) \tilde F_{nh}(T_i)\left\{1 - \tilde F_{nh}(T_i) \right\}$$
	Therefore, along almost all sequences $(T_1,\dd_1),(T_2,\dd_2)\dots$,
	\begin{align*}
	\sum_{i=1}^{n} \sigma_{nh,i}^2 &= n^{-1/5}\int \psi_{t,h}^2(u) \tilde F_{nh}(u)\left\{1 - \tilde F_{nh}(u) \right\}d\G_n(u)\\
	& = n^{-1/5}\int \psi_{t,h}^2(u) \tilde F_{nh}(u)\left\{1 - \tilde F_{nh}(u) \right\}dG(u) + o(1) \\
	&= \int_{-1}^1 K^2(u) \tilde F_{nh}(t + hu )\left\{1 - \tilde F_{nh}(t + hu) \right\}g(t+hu)du + o(1)\\
	& \to \frac{F_0(t)\{1-F_0(t)\}}{cg(t)}\int K^2(u)du=\sigma^2.
	\end{align*}
	where we use the a.s. convergence of $ F_{nh}(t) \to  F_{0}(t)$ in the last line. By the Lindeberg-Feller CLT, we have,
	\begin{align*}
	\sum_{i=1}^{n} Z_{nh,i} \stackrel{{\cal D}}{\longrightarrow}N\left(0,\sigma^2\right) .
	\end{align*}
	This proves (\ref{term1}).
	
	We next consider $A_{II}$. From the fact that the integrand is the product of $h^{-1}$ times the fixed bounded continuous function
	$u\mapsto K((t-u)/h)/g(u)$ and the class of functions of bounded variation $\hat F_n^*-\tilde F_{nh}$ which have entropy with bracketing of order $\e^{-1}$ for the $L_2$-distance and are of order $O_p^*(n^{-1/3})$ for the $L_2$-distance, again conditionally on $\omega=(T_1,\dd_1),(T_2,\dd_2),\dots$, it follows that $A_{II}$ is of order $O_p^*(h^{-1}n^{-2/3})$. As a consequence, we have for $h\asymp n^{-1/5}$,
	\begin{align}
	\label{term2}
	A_{II}= \int \psi_{t,h}(u) \left\{\tilde F_{nh}(u) - \hat F_{n}^*(u) \right\} d(\G_n-G)(u) = o_p^*(n^{-2/5})
	\end{align}
	
	We finally study the term $A_{III}$.  Using similar arguments as in the proof of Lemma  A.4 in \cite{piet_geurt_birgit:10}, there exists a positive constant C such that
	\begin{align}
	\label{diff-barpsi-psi}
	\left\vert \bar \psi_{t,h}^*(u)  - \psi_{t,h}(u) \right\vert \leq Ch^{-2}\left\vert \hat F_n^*(u)- \tilde F_{nh}(u) \right\vert
	\end{align} 
	for all $u$ such that $\tilde f_{nh} =\tilde F'_{nh}$ is positive and continuous in a neighborhood around $u$.
	By (\ref{expectation-delta}), we can write,
	\begin{align}
	\label{A_{III}-expansion}
	A_{III} &= \int \left\{\bar\psi_{t,h}^*(u) - \psi_{t,h}(u) \right \} \left\{\d^* - \tilde F_{nh}(u) \right\}d(\P_n^*-P_n^*)(u,\d^*)\nonumber\\
	&\qquad  + \int \left\{\bar\psi_{t,h}^*(u) - \psi_{t,h}(u) \right \} \left\{\tilde F_{nh}(u) - \hat F_{n}^*(u) \right\}d\G_n(u).
	\end{align}
	It is clear that
	\begin{align*}
	&\int \left\{\bar\psi_{t,h}^*(u) - \psi_{t,h}(u) \right \} \left\{\d^* - \tilde F_{nh}(u) \right\}d(\P_n^*-P_n^*)(u,\d^*)\\
	&\qquad=o_p^*\left(\int \psi_{t,h}(u) \left\{\d^* - \tilde F_{nh}(u) \right\}d(\P_n^*-P_n^*)(u,\d^*)\right),
	\end{align*}
	which is $o_p^*(n^{-2/5})$ by (\ref{term1}).
	For the second term on the right-hand side of (\ref{A_{III}-expansion}) we get by (\ref{diff-barpsi-psi}) and (\ref{local_L_2_bound}):
	\begin{align}
	\label{term3}
	& \left|\int \left\{\bar\psi_{t,h}^*(u) - \psi_{t,h}(u) \right \} \left\{\tilde F_{nh}(u) - \hat F_{n}^*(u) \right\}d\G_n(u)\right|\nonumber\\
	&\le Ch^{-2}\int_{t-h}^{t+h}\left\{\tilde F_{nh}(u) - \hat F_{n}^*(u) \right\}^2\,d\G_n(u)=O_p^*\left(h^{-1}n^{-2/3}\right)=O_p^*\left(n^{-7/15}\right).
	\end{align}
	The proof of Theorem \ref{th:bootstrap_SMLE} now follows by (\ref{term1}),(\ref{term2}) and (\ref{term3}).
\end{proof}

In the next section we give the proof of Lemma \ref{lemma:boundary_correction} about the boundary corrected version of the SMLE.
\subsection{Proof of Lemma \ref{lemma:boundary_correction}}
\begin{proof}
	We have:
	\begin{align*}
	&\int \left\{\IK_h(t-u)+\IK_h(t+u)-\IK_h(2M-t-u)\right\}\,d\tilde F_{nh}^{(bc)}(u)\\
	&=\int_{u=0}^M \left\{\IK_h(t-u)+\IK_h(t+u)-\IK_h(2M-t-u)\right\}\tilde f_{nh}^{(bc)}(u)\,du.
	\end{align*}
	If $t\in[h,M-h]$ we get, noting that $\IK_h(t+u)=\IK_h(2M-t-u)=1$, if $t\in[h,M-h]$,
	\begin{align*}
	&\int_{u=0}^{M} \left\{\IK_h(t-u)+\IK_h(t+u)-\IK_h(2M-t-u)\right\}\tilde f_{nh}^{(bc)}(u)\,du\\
	&=\int_{u=0}^{M} \IK_h(t-u)\tilde f_{nh}^{(bc)}(u)\,du\\
	&=\int_{u=0}^{M} \IK_h(t-u)\int\left\{K_h(u-v)+K_h(u+v)+K_h(2M-u-v)\right\}\,d\hat F_n(v)\,du\\
	&=\int\left\{\int_{u=0}^{M} \IK_h(t-u)\left\{K_h(u-v)+K_h(u+v)+K_h(2M-u-v)\right\}\,du\right\}\,d\hat F_n(v)\\
	&=\int\left\{\widetilde{\IK}_h(t-v)+\widetilde{\IK}_h(t+v)-\widetilde{\IK}_h(2M-t-v)\right\}\,d\hat F_n(v)
	\end{align*}
	The last transition follows from integration by parts and the symmetry of the kernel $K$:
	\begin{align*}
	&\int_{u=0}^{M} \IK_h(t-u)\left\{K_h(u-v)+K_h(u+v)+K_h(2M-u-v)\right\}\,du\\
	&=\left[\IK_h(t-u)\left\{\IK_h(u-v)+\IK_h(u+v)-\IK_h(2M-u-v)\right\}\right]_{u=0}^M\\
	&\qquad\qquad+\int K_h(t-u)\left\{\IK_h(u-v)+\IK_h(u+v)-\IK_h(2M-u-v)\right\}\,du\\
	&=\int K_h(t-u)\left\{\IK_h(u-v)+\IK_h(u+v)-\IK_h(2M-u-v)\right\}\,du\\
	&=\int \left\{\IK_h(t-v-hw)+\IK_h(t+v-hw)-\IK_h(2M-t-v-hw)\right\}K(w)\,dw\\
	&=\widetilde{\IK}_h(t-v)+\widetilde{\IK}_h(t+v)-\widetilde{\IK}_h(2M-t-v).
	\end{align*}
	if $t\in[h,M-h]$.
	
	We likewise get, if $t\in[0,h]$,
	\begin{align*}
	&\int_{u=0}^{M} \left\{\IK_h(t-u)+\IK_h(t+u)-\IK_h(2M-t-u)\right\}\tilde f_{nh}^{(bc)}(u)\,du\\
	&=\int_{u=0}^{M} \left\{\IK_h(t-u)+\IK_h(t+u)-1\right\}\tilde f_{nh}^{(bc)}(u)\,du\\
	\end{align*}
	\begin{align*}
	&=\int_{u=0}^{M} \left\{\IK_h(t-u)+\IK_h(t+u)-1\right\}\\
	&\qquad\qquad\qquad\qquad\cdot\int\left\{K_h(u-v)+K_h(u+v)+K_h(2M-u-v)\right\}\,d\hat F_n(v)\,du\\
	&=\int\left\{\widetilde{\IK}_h(t-v)+\widetilde{\IK}_h(t+v)-\widetilde{\IK}_h(2M-t-v)\right\}\,d\hat F_n(v).
	\end{align*}
	In the last transition we use integration by parts again:
	\begin{align*}
	&\int_{u=0}^{M} \left\{\IK_h(t-u)+\IK_h(t+u)-1\right\}\\
	&\qquad\qquad\qquad\qquad\cdot\left\{K_h(u-v)+K_h(u+v)+K_h(2M-u-v)\right\}\,du\\
	&=\left[\left\{\IK_h(t-u)+\IK_h(t+u)-1\right\}\left\{\IK_h(u-v)+\IK_h(u+v)-\IK_h(2M-u-v)\right\}\right]_{u=0}^M\\
	&\qquad+\int_{u=0}^M\left\{K_h(t-u)-K_h(t+u)\right\}\left\{\IK_h(u-v)+\IK_h(u+v)-\IK_h(2M-u-v)\right\}\,du\\
	&=\int_{u=0}^M\left\{K_h(t-u)-K_h(t+u)\right\}\left\{\IK_h(u-v)+\IK_h(u+v)-\IK_h(2M-u-v)\right\}\,du,
	\end{align*}
	where we use $\IK_h(-v)+\IK_h(v)=1$ in the last equality (which follows from the symmetry of $K$).
	Furthermore,
	\begin{align*}
	&\int_{u=0}^M\left\{K_h(t-u)-K_h(t+u)\right\}\left\{\IK_h(u-v)+\IK_h(u+v)-\IK_h(2M-u-v)\right\}\,du\\
	&=\int_{w=-1}^{t/h}K(w)\left\{\IK_h(t-v-hw)+\IK_h(t+v-hw)-1\right\}\,dw\\
	&\qquad\qquad-\int_{w=t/h}^{1}K(w)\left\{\IK_h(-t-v+hw)+\IK_h(-t+v+hw)-1\right\}\,dw\\
	&=\int_{w=-1}^{t/h}K(w)\left\{\IK_h(t-v-hw)+\IK_h(t+v-hw)-1\right\}\,dw\\
	&\qquad\qquad+\int_{w=t/h}^{1}K(w)\left\{\IK_h(t+v-hw)+\IK_h(t-v-hw)-1\right\}\,dw\\
	&=\int_{w=-1}^1K(w)\left\{\IK_h(t-v-hw)+\IK_h(t+v-hw)-1\right\}\,dw\\
	&=\int K(w)\left\{\IK_h(t-v-hw)+\IK_h(t+v-hw)-K_h(2M-t-v-hw)\right\}\,dw\\
	&=\widetilde{\IK}_h(t-v)+\widetilde{\IK}_h(t+v)-\widetilde{\IK}_h(2M-t-v),
	\end{align*}
	again using the relation $\IK_h(x)+\IK_h(-x)=1$.

	The case $t\in[M-h,M]$ is treated similarly.
\end{proof}
In the remaining subsection we prove Lemma \ref{lemma_Th11.3} needed in the proof of Theorem \ref{th:bootstrap_SMLE}.

\subsection{Proof of Lemma \ref{lemma_Th11.3}}
In the proof of Lemma \ref{lemma_Th11.3} we use the following (Dvoretsky-Kiefer-Wolfowitz-type) inequality from \cite{mouli:16}.

\begin{lemma}[Lemma 8.1 of \cite{mouli:16}]
	\label{lem: DKWinv} 
	Let $F$ be a distribution function on $\R$ with a density $f$ supported on $[0,1]$ and bounded away from zero on $[0,1]$.
	Let $\F_{n}$ be the empirical distribution function associated with a sample of $n$ observations from  $F$ and let $\F_{n}^{-1}$ be the corresponding empirical quantile function. With $c$ a lower bound for $f$, we then have
	$$
	\P\left(\sup_{t\in[0,1]}|\F_{n}^{-1}(t)-F^{-1}(t)|>x\right)\leq 4\exp(-2nc^2x^2)
	$$
	for all $n$ and $x>0$. 
\end{lemma}

\begin{proof}[Proof of Lemma \ref{lemma_Th11.3}] We follow notation, introduced in Section 4.1 of \cite{mouli:16}, but now applied to a bootstrap sample $(T_1,\dd_1^*)\dots,(T_n,\dd_n^*)$. Just as in the proof of the corresponding Theorem 11.3 in \cite{piet_geurt:14}, Doob's inequality and exponential
	centering play an important role in the proof.

	Moreover, we prove the equivalent statement
	\begin{align}
	\label{equiv_statement}
	P_n^*\left\{\bigl|U_n^*(a)-U_n(a)\bigr|>x\right\}\le c_1\exp\left\{-c_2nx^3\right\},
	\end{align}
	almost surely, for all large $n$, and constants $c_1,c_2>0$ and all $x\in(n^{-1/3},M]$. To see that this is equivalent, first note that 
	\begin{align*}
	P_n^*\left\{n^{1/3}\bigl|U_n^*(a)-U_n(a)\bigr|>x\right\}=P_n^*\left\{\bigl|U_n^*(a)-U_n(a)\bigr|> n^{-1/3}x\right\},
	\end{align*}
	so, if (\ref{equiv_statement}) holds, we get:
	\begin{align*}
	P_n^*\left\{n^{1/3}\bigl|U_n^*(a)-U_n(a)\bigr|>x\right\}\le c_1\exp\left\{-c_2x^3\right\},
	\end{align*}
	for all $x>0$.	
	Next note that for $x\in[0,n^{-1/3}]$
	$$
	c_1\exp\left\{-c_2nx^3\right\}\ge c_1\exp\left\{-c_2\right\}\ge1,
	$$
	if $c_1\ge e^{c_2}$. So we can always adapt the constants in such a way that the inequality is satisfied for $x\in[0,n^{-1/3}]$.
	
	Furthermore, for $x\in[1,M]$, we can write:
	\begin{align*}
	c_1\exp\left\{-c_2nx^2\right\}\le c_1\exp\left\{-(c_2/M)nx^3\right\}
	\end{align*}
	So for $x\in[1,M]$, we only need an inequality with $x^2$ in the exponent on the right-hand side, and can use Lemma \ref{lem: DKWinv} 
	to our advantage (see below). Finally, for $x>M$, the probability on the left-hand side of (\ref{equiv_statement}) is zero.
	
	Let $\Lambda_n^*:[0,1]\to[0,1]$ be defined by $\Lambda_n^*(0)=0$, and
	\begin{align*}
	\Lambda_n^*(i/n)=n^{-1}\sum_{j\le i}\dd_j^*,\qquad i=1,\dots,n,
	\end{align*}
	and by linear interpolation at other points of $[0,1]$. Furthermore, let $\l_n^*$ be the left-continuous slope of the greatest convex minorant of $\Lambda_n^*$.
	Then:
	$$
	\hat F_n^*(T_i)=\l_n^*(i/n)=\l_n^*(\G_n(T_i)),
	$$
	where $\G_n$ is the empirical distribution function of the observations $T_1,\dots,T_n$ and $\hat F_n^*$ is the MLE in the bootstrap sample.
	
	We define analogously $\tilde \l_n=\tilde F_{nh}\circ G^{-1}$, and
	\begin{align*}
	\tilde \L_n(t)=\int_0^t\tilde\l_n(u)\,du=\int_0^t \tilde F_{nh}\left(G^{-1}(u)\right)\,du,\qquad t\in[0,1].
	\end{align*}
	Moreover, we define:
	\begin{align}
	\label{def_V_n}
	V_n=\tilde\l_n^{-1}.
	\end{align}
	
	With these definitions we have:
	\begin{align}
	\label{U_n_V_n}
	U_n=G^{-1}\circ\tilde\l_n^{-1}=G^{-1}\circ V_n,
	\end{align}
	where $U_n$ is defined by (\ref{def_U_n}). 
	By the model assumptions at the beginning of Section \ref{section:pointwiseCI} for $F_0$ and $G$, and the almost sure convergence of $\tilde F_{nh}$ and its derivative to $F_0$ and $f_0$, respectively, uniformly on $[0,M]$ (using the suggested boundary correction near $0$ and $M$), we may assume that there is a constant $c>0$ such that $\tilde \l_n'(t)\ge c$ for all $t\in[0,1]$ and all large $n$, and that therefore, using a Taylor expansion, we get:
	\begin{align}
	\label{lower_bound_L_n}
	\tilde \L_n(t)-\tilde \L_n(V_n(a))\ge \bigl(t-V_n(a)\bigr) a+\tfrac12 c\bigl(t-V_n(a)\bigr)^2,
	\end{align}
	for all $t,a\in[0,1]$.
	
	We similarly define
	\begin{align*}
	V_n^*(a)=\argmin_{u\in[0,1]}\{\Lambda_n^*(u)-a u\},
	\end{align*}
	where $\argmin$ denotes the smallest location of the minimum. Note that, analogously to (\ref{U_n_V_n}), we have for $U_n^*$ as defined by (\ref{argmin_process}):
	\begin{align}
	\label{U_n*_V_n^*}
	U_n^*=\G_n^{-1}\circ V_n^*.
	\end{align}
	By the transition of $U_n$ and $U_n^*$ to $V_n$ and $V_n^*$, respectively, the range of $U_n$ and $U_n^*$  is changed from $[0,M]$ to $[0,1]$. We now prove:
	\begin{align}
	\label{equiv_statement2}
	P_n^*\left\{\bigl|V_n^*(a)-V_n(a)\bigr|>x\right\}\le c_1\exp\left\{-c_2nx^3\right\},
	\end{align}
	almost surely, for all large $n$, and constants $c_1,c_2>0$ and all $x\in(n^{-1/3},1]$. Note that the probability on the left-hand side of (\ref{equiv_statement2}) is zero if $x>1$.

	Define
	\begin{align*}
	\e_i^*=\dd_i^*-\tilde F_{nh}(T_i),\qquad i=1,\dots,n.
	\end{align*}
	Then:
	\begin{align*}
	\L_n^*(i/n)&=n^{-1}\sum_{j\le i}\e_j^*+n^{-1}\sum_{j\le i}\tilde F_{nh}\left(\G_n^{-1}(j/n)\right)\\
	&=n^{-1}\sum_{j\le i}\e_j^*+\int_0^{i/n}\tilde F_{nh}\left(\G_n^{-1}(u)\right)\,du,\qquad i=1,\dots,n,
	\end{align*}
	using the piecewise constancy of $\G_n^{-1}$.
	
	This gives:
	\begin{align*}
	&P_n^*\left\{\bigl|V_{n}^*(a)-V_n(a)\bigr|>x\right\}\\
	&\leq P_n^*\left\{\min_{i:\ |V_n(a)-i/n|>x}\{\Lambda_{n}^*(i/n)-a\,i/n\}\leq\L^*_{n}(V_n(a))-a V_n(a)\right\}\\
	&\leq P_n^*\left\{\min_{i:\ |V_n(a)-i/n|>x}\left\{D_{n}^*(i/n)- D_{n}^*\left(V_n(a)\right)+\tfrac12 c\left(in^{-1}-V_n(a)\right)^2\right\}\leq 0\right\},
	\end{align*}
	where $D_n^*$ is defined by $D_n^*=\L_n^*-\tilde\L_n$ and where we use (\ref{lower_bound_L_n}) in the last step.
	Define
	\begin{align*}
	B_n^*(t)=D_n^*(t)-\int_0^t\left\{\tilde F_{nh}\left(\G_n^{-1}(u)\right)-\tilde F_{nh}\left(G^{-1}(u)\right)\right\}\,du.
	\end{align*}
	Then:
	\begin{align*}
	B_n^*(i/n)=n^{-1}\sum_{j\le i}\e_i^*.
	\end{align*}
	Moreover, the event $\{\bigl|V_{n}^*(a)-V_n(a)\bigr|>x\}$ is contained in the union of the events
	\begin{align*}
	E_{n1}=\left\{\sup_{u:\ |V_n(a)-u|>x}\left\{\int_{V_n(a)}^u\left\{\tilde F_{nh}\left(G^{-1}(t)\right)-\tilde F_{nh}\left(\G_n^{-1}(t)\right)\right\}\,dt-\frac c4(u-V_n(a))^2\right\}\geq 0\right\}
	\end{align*}
	and
	\begin{align*}
	E_{n2}=\left\{\sup_{i:\ |V_n(a)-i/n|>x}\{B_n^*(V_n(a))-B_n^*(i/n)-\frac c4(in^{-1}-V_n(a))^2\}\geq 0\right\}.
	\end{align*}
	We have, by the mean value theorem and the bounded differentiability of $\tilde F_{nh}$,
	\begin{align*}
	&\left|\int_{V_n(a)}^u\left\{\tilde F_{nh}\left(G^{-1}(t)\right)-\tilde F_{nh}\left(\G_n^{-1}(t)\right)\right\}\,dt\right|
	\le c'\bigl|u-V_n(a)\bigr|\sup_{t\in[0,1]}\bigl|\G_n^{-1}(t)-G^{-1}(t)\bigr|.
	\end{align*}
	for a constant $c'>0$. Hence we get from Lemma \ref{lem: DKWinv} in the original space:
	\begin{align}
	\label{bound-E_n1}
	&P_n(E_{n1})\le P_n\left\{\sup_{t\in[0,1]}\bigl|\G_n^{-1}(t)-G^{-1}(t)\bigr|\ge \frac{c x}{4c'}\right\}\le 4\exp\left\{-K nc^2x^2\right\}
	\le 4\exp\left\{-Kc^2 n^{1/3}\right\},
	\end{align}
	for some $K>0$ and $x\in(n^{-1/3},M]$.	
	This means that we may assume that, almost surely, the complement of $E_{n1}$ is satisfied for all large $n$ and all $x\in(n^{-1/3},1]$.
	So we now turn to $P_n^*(E_{n2})$.
	
	We have:
	\begin{align*}
	P_n^*(E_{n2})&\leq \sum_{k\geq 1}P_n^*\left(\sup_{i:\ |V_n(a)-i/n|\in(kx,(k+1)x]}\left\{B_n^*\left(V_n(a)\right)-B_n^*(i/n)-\frac c4\left(i/n-V_n(a)\right)^2\right\}\geq 0\right)\\
	&\leq \sum_{k\geq 1}P_n^*\left(\sup_{i:\ |V_n(a)-i/n|\leq(k+1)x}\{B_n^*\left(V_n(a)\right)-B_n^*(i/n)\}\geq \frac c4k^2x^2\right).
	\end{align*}
	Using the piecewise linearity of $B_n^*$, we get
	$$
	B_n^*\left(V_n(a)\right)=B_n^*\left(\frac{\lfloor n V_n(a) \rfloor}n\right)+\left(V_n(a)-\frac{\lfloor n V_n(a) \rfloor}n\right)\e^*_{\lfloor n V_n(a) \rfloor+1},
	$$
	where $\lfloor n V_n(a) \rfloor$ denotes the integer part (``floor'') of $n V_n(a)$. Hence,
	\begin{eqnarray}
	\label{eq: majS1S2} \notag
	P_n^*(E_{n2})
	&\leq& \sum_{k\geq 1}P_n^*\left(\left(V_n(a)-\frac{\lfloor nV_n(a) \rfloor}n\right)\e^*_{\lfloor nV_n(a) \rfloor+1}\geq \frac c8k^2x^2\right)\\
	&&\qquad+\sum_{k\geq 1}P_n^*\left(\sup_{i:\ |V_n(a)-i/n|\leq (k+1)x}\left\{\sum_{j\leq nV_n(a)}\e^*_{j}-\sum_{j\leq i}\e^*_{j}\right\}\geq \frac {nc}8k^2x^2\right).
	\end{eqnarray}
	
	The Markov inequality implies that for all $\theta>0$, $k\geq 1$, $a\in[0,1]$ and $x\in(n^{-1/3},1]$,
	\begin{equation}
	\notag
	\begin{split}
	&   P_n^*\left\{\left(V_n(a)-\frac{\lfloor nV_n(a) \rfloor}n\right)\e^*_{\lfloor nV_n(a) \rfloor+1}\geq \frac c8k^2x^2\right\}\\
	&\leq \exp\left\{- \frac { \theta c}8k^2x^2\right\} E_n^*\exp\left\{\theta\left(V_n(a)-\frac{\lfloor nV_n(a) \rfloor}n\right)\e^*_{\lfloor nV_n(a) \rfloor+1}\right\},
	\end{split}
	\end{equation}
	where $E_n^*$ denotes the expectation under $P_n^*$.
	Since $\e^*_i\in[-1,1]$ for all $i$, we have $\exp( \alpha \e^*_i)\leq K\exp(\alpha^2)$ for all $\alpha\in\R$ and $K\ge \exp(1)$ and therefore, with $\theta=ck^2x^2n^2/16$, we obtain
	\begin{eqnarray*}
		P_n^*\left(\left(V_n(a)-\frac{\lfloor nV_n(a) \rfloor}n\right)\e^*_{\lfloor nV_n(a) \rfloor+1}\geq \frac c8k^2x^2\right) &\leq &K\exp\left(- \frac{ \theta c}8k^2x^2+\frac{\theta^2}{n^2}\right) \\
		& \leq& K\exp\left(- \frac{ c^2k^4x^4n^2}{256}\right) .
	\end{eqnarray*}
	Using that $k^4\geq k$ for all $k\geq 1$ and $nx\geq 1$ for all $x\in(n^{-1/3},1)$, we conclude that for all $a\in[0,1]$ and $x\in(n^{-1/3},1)$
	\begin{align}
	\notag \label{eq: S1}
	&\sum_{k\geq 1}P_n^*\left(\left(V_n(a)-\frac{\lfloor nV_n(a) \rfloor}n\right)\e^*_{\lfloor nV_n(a) \rfloor+1}\geq \frac c8k^2x^2\right)\nonumber\\
	&\leq  K\sum_{k\geq 1}\exp\left(- \frac{ c^2kx^3n}{256}\right)
	\leq  K\exp\left(- \frac{ c^2x^3n}{256}\right)  \sum_{k\geq 0}\exp\left(- \frac{ c^2kx^3n}{256}\right)\nonumber\\
	&\leq K'\exp(-K_{2}nx^3),
	\end{align}
	with any finite $K'$ that satisfies
	$K'\geq K\sum_{k\geq 0}\exp\left(-c^2k/256\right)$
	and $K_{2}\leq c^2/256$. This takes care of the first term on the right of (\ref{eq: majS1S2}).
	
	We now consider the second term on the right of (\ref{eq: majS1S2}). Just as in the proof of Theorem 11.3 in \cite{piet_geurt:14}, we use Doob's submartingale inequality, this time conditionally on $(T_1,\dd_1),\dots,(T_n,\dd_n)$. This gives:
	\begin{equation}\notag
	\begin{split}
	&P_n^*\left(\sup_{i:\ |V_n(a)-i/n|\leq (k+1)x}\left\{\sum_{j\leq nV_n(a)}\e^*_{j}-\sum_{j\leq i}\e^*_{j}\right\}\geq \frac {nc}8k^2x^2\right)\\
	&\leq 2 \exp\left(-\frac {\theta nc}8k^2x^2\right)\sup_{i:\ |V_n(a)-i/n|\leq (k+1)x} E_n^*\left[\exp\left(\theta \left(\sum_{j\leq nV_n(a)}\e^*_{j}-\sum_{j\leq i}\e^*_{j}\right)\right)\right].
	\end{split}
	\end{equation}
	Suppose $i/n< V_n(a)$. Then we get:
	\begin{align*}
	&\log E_n^*\left[\exp\left(\theta \left(\sum_{j\leq nV_n(a)}\e^*_{j}-\sum_{j\leq i}\e^*_{j}\right)\right)\right]
	=\log E_n^*\left[\exp\left(\theta \left(\sum_{i<j\leq nV_n(a)}\e^*_{j}\right)\right)\right]\\
	&=\sum_{i<j\leq nV_n(a)}
	\log\left\{\exp\left\{\th\{1-\tilde F_{nh}(T_j)\}\right\}\tilde F_{nh}(T_j)
	+\exp\left\{-\th \tilde F_{nh}(T_j)\right\}\{1-\tilde F_{nh}(T_j)\}\right\}\\
	&=n\int_{i/n}^{V_n(a)}\log\Biggl\{\exp\left\{\th\{1-\tilde F_{nh}(\G_n^{-1}(t))\}\right\}\tilde F_{nh}(\G_n^{-1}(t))\\
	&\qquad\qquad\qquad\qquad\qquad\qquad\qquad+\exp\left\{-\th\tilde F_{nh}(\G_n^{-1}(t))\}\right\}\left\{1-\tilde F_{nh}(\G_n^{-1}(t))\right\}\Biggr\}\,dt
	\end{align*}
	Since $\log(1+x)\le x$, this is bounded above by:
	\begin{align*}
	&n\int_{i/n}^{V_n(a)}\Biggl\{\exp\left\{\th\{1-\tilde F_{nh}(\G_n^{-1}(t))\}\right\}\tilde F_{nh}(\G_n^{-1}(t))\\
	&\qquad\qquad\qquad\qquad\qquad\qquad\qquad+\exp\left\{-\th\tilde F_{nh}(\G_n^{-1}(t))\}\right\}\left\{1-\tilde F_{nh}(\G_n^{-1}(t))\right\}-1\Biggr\}\,dt\\
	&\le
	n\int_{V_n(a)-(k+1)x}^{V_n(a)}\Biggl\{\exp\left\{\th\{1-\tilde F_{nh}(\G_n^{-1}(t))\}\right\}\tilde F_{nh}(\G_n^{-1}(t))\\
	&\qquad\qquad\qquad\qquad\qquad\qquad\qquad+\exp\left\{-\th\tilde F_{nh}(\G_n^{-1}(t))\}\right\}\left\{1-\tilde F_{nh}(\G_n^{-1}(t))\right\}-1\Biggr\}\,dt,
	\\
	&=
	n\int_{V_n(a)-(k+1)x}^{V_n(a)}\Biggl\{\sum_{i=2}^{\infty}\frac{\th^i}{i!}\{1-\tilde F_{nh}(\G_n^{-1}(t))\}^i\tilde F_{nh}(\G_n^{-1}(t))\\
	&\qquad\qquad\qquad\qquad\qquad\qquad\qquad+\sum_{i=2}^{\infty}\frac{\th^i}{i!}(-1)^i\tilde F_{nh}(\G_n^{-1}(t))^i\left\{1-\tilde F_{nh}(\G_n^{-1}(t))\right\}\Biggr\}\,dt,
	\\
	&=
	n \sum_{i=2}^{\infty}\frac{\th^i}{i!} \int_{V_n(a)-(k+1)x}^{V_n(a)}\Biggl\{\{1-\tilde F_{nh}(\G_n^{-1}(t))\}^i\tilde F_{nh}(\G_n^{-1}(t))\\
	&\qquad\qquad\qquad\qquad\qquad\qquad\qquad+(-1)^i\tilde F_{nh}(\G_n^{-1}(t))^i\left\{1-\tilde F_{nh}(\G_n^{-1}(t))\right\}\Biggr\}\,dt,
	\end{align*}
	if $i/n<V_n(a)$ and $|V_n(a)-i/n|\leq (k+1)x$. Since $V_n(a) \in [0,1]$, the integrand,
	\begin{align*}
	\{1-\tilde F_{nh}(\G_n^{-1}(t))\}^i\tilde F_{nh}(\G_n^{-1}(t)) + (-1)^i\tilde F_{nh}(\G_n^{-1}(t))^i\left\{1-\tilde F_{nh}(\G_n^{-1}(t))\right\},
	\end{align*}
	is bounded by 1/2, we get,
	\begin{align*}
	&P_n^*\left(\sup_{i:\ |V_n(a)-i/n|\leq (k+1)x}\left\{\sum_{j\leq nV_n(a)}\e^*_{j}-\sum_{j\leq i}\e^*_{j}\right\}\geq \frac {nc}8k^2x^2\right)\\
	&\leq 2 \exp\left(-\frac {\theta nc}8k^2x^2\right)\sup_{i:\ |V_n(a)-i/n|\leq (k+1)x} E_n^*\left[\exp\left(\theta \left(\sum_{j\leq nV_n(a)}\e^*_{j}-\sum_{j\leq i}\e^*_{j}\right)\right)\right]\\
	&\leq 2 \exp\left(-\frac {\theta nck^2x^2}8  + \frac {n(k+1)x}2 \sum_{i=2}^{\infty}\frac{\th^i}{i!}  \right)
	\end{align*}
	for all $x\in (n^{-1/3},1),k \geq 1, a\in[0,1]$ and $\th >0$. Therefore, with $\th = \log(1+\frac{ck^2x}{4(k+1)})$, we arrive at,
	\begin{align*}
	&P_n^*\left(\sup_{i:\ |V_n(a)-i/n|\leq (k+1)x}\left\{\sum_{j\leq nV_n(a)}\e^*_{j}-\sum_{j\leq i}\e^*_{j}\right\}\geq \frac {nc}8k^2x^2\right)\\
	&\leq 2 \exp\left(\frac {n(k+1)x}2\left\{\frac{ck^2x}{4(k+1)} - \left(1+ \frac{ck^2x}{4(k+1)} \right)\log\left(1+\frac{ck^2x}{4(k+1)}\right)  \right\} \right)
	\end{align*}
	Following \cite{Pol:84}, in his discussion of Bennett's inequality on p. 192, we introduce the function $B$, defined by $B(0)=1/2$ and 
	\begin{align*}
	B(u) = u^{-2}\{ (1+u)\log(1+u) - u \}.
	\end{align*}	
	Making the change of variables $u_k = ck^2x/(4(k+1))$, we can write,
	\begin{align*}
	&P_n^*\left(\sup_{i:\ |V_n(a)-i/n|\leq (k+1)x}\left\{\sum_{j\leq nV_n(a)}\e^*_{j}-\sum_{j\leq i}\e^*_{j}\right\}\geq \frac {nc}8k^2x^2\right)\leq 2 \exp\left(-\frac{nc^2k^4x^3}{32(k+1)}B(u_k) \right).
	\end{align*}
	Since $u_k$ varies over a finite interval $[0,M']$ and therefore $B(u_k)$ stays away from zero on $[0,M']$, we find that,
	\begin{align*}
	&\sum_{k\geq 1}P_n^*\left(\sup_{i:\ |V_n(a)-i/n|\leq (k+1)x}\left\{\sum_{j\leq nV_n(a)}\e^*_{j}-\sum_{j\leq i}\e^*_{j}\right\}\geq \frac {nc}8k^2x^2\right)
	\\
	&\leq K_1\sum_{k\geq 1} \exp\left(-\frac{nc^2k^4x^3}{32(k+1)} \right)
	\leq K_1 \exp\left(-\frac{nc^2x^3}{64}\right)\sum_{k\geq 0} \exp\left(-\frac{c^2k^4}{32(k+1)} \right)
	\leq K_2 \exp\left(-K_3nx^3\right).
	\end{align*}
	for appropriate $K_1, K_2$ and $K_3$. Combining this with (\ref{bound-E_n1}) and (\ref{eq: S1}), it follows that 
	\begin{align*}
	P_n^*\left\{\bigl|V_n^*(a)-V_n(a)\bigr|>x\right\}\le c_1\exp\{-nc_2x^3) 
	\end{align*}
	for all large $n$, almost surely along $(T_1,\Delta_1),\ldots $
	for  constants $c_1,c_2>0$ and $x\in (n^{-1/3},1]$.
	
	We now prove that (\ref{equiv_statement}) also follows by considering the transition  of $V_n$ and $V_n^*$ to $U_n$ and $U_n^*$. By (\ref{U_n_V_n}) and (\ref{U_n*_V_n^*}) we get:
	\begin{align*}
	U_n^*(a)-U_n(a)=\G_n^{-1}\circ V_n^*(a)-G^{-1}\circ V_n(a),
	\end{align*}
	and hence:
	\begin{align*}
	\bigl|U_n^*(a)-U_n(a)\bigr|
	\le \sup_{t\in[0,1]}\bigl|\G_n^{-1}(t)-G^{-1}(t)\bigr|+k_1\bigl|V_n^*(a)-V_n(a)\bigr|,
	\end{align*}
	where
	\begin{align*}
	k_1=1/\inf_{x\in[0,M]}g(x).
	\end{align*}
	From Lemma \ref{lem: DKWinv} we get in the original space:
	\begin{align*}
	P_n\left\{\sup_{t\in[0,1]}\bigl|\G_n^{-1}(t)-G^{-1}(t)\bigr|\ge x/2\right\} 
	\le 4\exp\left\{-K n^{1/3}\right\},
	\end{align*}
	for some $K>0$ and $x\in(n^{-1/3},M]$.	
	So we may assume that, almost surely, $\bigl|\G_n^{-1}(t)-G^{-1}(t)\bigr|< x/2$, for all large $n$ and all $x\in(n^{-1/3},1]$.
	By the foregoing proof, we also have:
	\begin{align*}
	P_n^*\left\{k_1\bigl|V_n^*(a)-V_n(a)\bigr|\ge x/2\right\}\le c_1\exp\left\{-c_2nx^3/\left(8k_1^3\right)\right\}.
	\end{align*}
	This proves the result.
\end{proof}

\bibliographystyle{apalike} 

\end{document}